\documentclass[a4paper]{amsart}
\usepackage{amssymb,amsmath,amsthm,amstext,amsfonts}
\usepackage[dvips]{graphicx}
\usepackage{psfrag}
\usepackage{url}
\usepackage{amsfonts}
\usepackage{amsmath}
\usepackage{mathrsfs}
\usepackage{graphicx}
\usepackage{amsmath,amsthm,amssymb,amscd}
\usepackage{color}
\usepackage{enumerate}
\usepackage[colorlinks, linkcolor=blue, anchorcolor=blue, citecolor=blue]{hyperref}

\usepackage{epsfig}

\makeatletter \@addtoreset{equation}{section} \makeatother

\renewcommand\thetable{\thesection.\@arabic\c@table}

\theoremstyle{plain}
\newtheorem{maintheorem}{Theorem}

\newtheorem{maincorollary}{Corollary}

\newtheorem{theorem}{Theorem}[section]

\newtheorem{lemma}{Lemma}[section]

\newtheorem{definition}{Definition}[section]
\newtheorem{remark}{Remark}[section]

\newtheorem{Thm}{Theorem}[section]
\newtheorem{Lem}[Thm]{Lemma}
\newtheorem{Prop}[Thm]{Proposition}
\newtheorem{Cor}[Thm]{Corollary}
\theoremstyle{remark}
\newtheorem{Def}[Thm] {Definition}
\newtheorem{Rem}[Thm] {Remark}

\newcommand{\eps}{\varepsilon}

\newcommand{\N}{{\mathbb{N}}}

\long\def\begcom#1\endcom{}

\newcommand{\length}{\operatorname{\length}}

\def\length{\operatorname{length}}

\newcommand{\bl} {\begin{lemma}}
\newcommand{\el} {\end{lemma}}
\newcommand{\bt} {\begin{theorem}}
\newcommand{\et} {\end{theorem}}

\newcommand{\bp}{\begin{proof}}
\newcommand{\ep}{\end{proof}}
\newcommand  {\ee} {\end{equation}}
\newcommand  {\beq} {\begin{eqnarray*}}
\newcommand  {\eeq} {\end{eqnarray*}}

\newcommand  {\bd} {\begin{definition}}
\newcommand  {\ed} {\end{definition}}

\newcommand{\cM}{\mathcal{M}}
\newcommand{\cB}{\mathcal{B}}

\def\ep{\noindent{\hfill $\Box$}}

\begin{document}

\title{Entropy of irregular points that are not uniformly hyperbolic}

\author{Xiaobo Hou and Xueting Tian}
\address{Xiaobo Hou, School of Mathematical Sciences,  Fudan University\\Shanghai 200433, People's Republic of China}
\email{20110180003@fudan.edu.cn}

\address{Xueting Tian, School of Mathematical Sciences,  Fudan University\\Shanghai 200433, People's Republic of China}
\email{xuetingtian@fudan.edu.cn}

\begin{abstract}
In this article we prove that for a $C^{1+\alpha}$ diffeomorphism on a compact Riemannian manifold, if there is a hyperbolic ergodic measure whose support is not uniformly hyperbolic, then the topological entropy of the set of irregular points that are not uniformly hyperbolic is larger than or equal to the metric entropy of the hyperbolic ergodic measure. In the process of proof, we give an abstract general mechanism to study topological entropy of irregular points provided that the system has a sequence of nondecreasing invariant compact subsets such that every subsystem has shadowing property and is transitive. 
\end{abstract}

\thanks
{X. Tian is the corresponding author.}
\maketitle
\section{Introduction}
Let $f:M \rightarrow M$ be a continuous map on a compact Riemannian manifold $M.$ From Birkhoff's ergodic theorem, the irregular set (also called points with historic behavior,
see \cite{Ruelle,Takens})
\begin{equation*}
	IR(f) := \left\{x\in M: \lim_{n\to\infty}\frac1n\sum_{i=0}^{n-1}\delta_{f^i(x)} \,\, \text{ diverges }\right\}
\end{equation*}
has zero measure for any invariant measure. But this does not mean that  irregular set is empty. In fact, for many systems it was shown to be large and have strong dynamical complexity.
In this paper we focus on the topological entropy of irregular sets. Historically, there has been a considerable amount of works on this. Pesin and Pitskel \cite{Pesin-Pitskel1984} are the first to notice the phenomenon of the irregular set carrying full topological entropy in the case of the full shift on two symbols. Barreira and Schmeling \cite{Barreira-Schmeling2000} studied a broad class of systems including conformal repellers and horseshoes and showed that the irregular set carries full entropy. Later on, Chen $\mathit{et\ al}$ \cite{CKS} considered systems with the speciﬁcation property and and showed that the irregular set carries full entropy. In \cite{DOT} Dong $\mathit{et\ al}$ considered (possibly not transitive) systems with the shadowing property and proved that irregular set carries full topological entropy if it's non-empty. There are also lots of advanced results to show that the irregular points can carry full entropy in systems with various specification-like or shadowing-like properties, for example, see \cite{Pesin1997, Thompson2008, LLST, TianVarandas}.
Apart from topological entropy, it has been proved that irregular sets have strong dynamical complexity in sense of Hausdorff dimension \cite{Barreira-Schmeling2000}, Lebesgue positive measure \cite{Takens, KS}, topological pressure \cite{Thompson2008}, distributional chaos \cite{CT} and residual property \cite{DTY2015,LW2014} etc.

In uniformly hyperbolic systems,  irregular set has strong dynamical complexity as shown by the existing results. In dynamics beyond uniform hyperbolicity, using Katok's approximation of hyperbolic measures by horseshoes, Dong and Tian obtain that the topological entropy of irregular set is bounded from below by metric entropies of ergodic hyperbolic measures.
\begin{Thm}\cite[Corollary B]{DT}\label{dongtian}
	Let $M$ be a compact Riemannian manifold  of dimension at least $2$ and $f$ be a $C^{1+\alpha}$ diffeomorphism.   Then one has
	\begin{equation*}
		\begin{split}
			h_{top}(f,IR(f))&\geq\sup\{h_{top}(f,\Lambda): \Lambda\text{ is a transitive locally maximal hyperbolic set}\}\\
			&=\sup\{h_\mu(f): \mu\text{ is a hyperbolic ergodic measure}\}.
		\end{split}
	\end{equation*}
\end{Thm}
In fact Theorem \ref{dongtian}  obtains that the topological entropy of irregular points which are uniformly hyperbolic is bounded from below by metric entropy of ergodic hyperbolic measures, where a point $x\in M$ is said to be uniformly hyperbolic if $\overline{\{f^n(x)\}_{n=0}^\infty}$ is a uniformly hyperbolic set, otherwise, we say $x$ is not uniformly hyperbolic. But Theorem \ref{dongtian} does not imply the topological entropy of irregular points which are not uniformly hyperbolic is positive, and even can not obtain whether the set of irregular points which are not uniformly hyperbolic is non-empty. Recently, Hou, Lin and Tian in \cite{HLT} show that for every non-trival homoclinic class $H(p)$, there is a residual invariant subset $Y\subset H(p)$ such that each $x\in Y$ is a irregular point and  the orbit of $x$ is dense in $H(p)$ (also see \cite{CarVar2021}). Note that every transitive point is not uniformly hyperbolic if $H(p)$ is not uniformly hyperbolic. Then we can see that the set of irregular points which are not uniformly hyperbolic is residual for every non-trival homoclinic class which is not uniformly hyperbolic. From the existence of irregular points which are not uniformly hyperbolic, a natural question arises: whether irregular points which are not uniformly hyperbolic have strong dynamical complexity in sense of Hausdorff dimension, Lebesgue positive measure, topological entropy, topological pressure, and distributional chaos etc?
Generally, people study dynamical complexity of irregular points by specification-like properties (for example see \cite{CKS,DOT}) or a ﬁltration of Pesin blocks (for example see \cite{LLST}). But for irregular points which are not uniformly hyperbolic, on one hand, there is not uniformly hyperbolicity, so we don't have specification-like properties in this set. On the other hand, a ﬁltration of Pesin blocks is just a set with full measure for some invariant measure, irregular points which are not uniformly hyperbolic may be not in the  ﬁltration of Pesin blocks. Fortunately, for homolcinic classes there is a ﬁltration of horseshoes in sense of topology which we find is a useful tool to study irregular points which are not uniformly hyperbolic. And in the present paper, for possibly more applications we give an abstract general mechanism to study irregular points provided that the system has a sequence of nondecreasing invariant compact subsets such that every subsystem has shadowing property and is transitive. 
In our previous paper \cite{HT}, we proved that the irregular points that are not uniformly hyperbolic are strongly distributional chaos.
In this paper, we will prove that the irregular points that are not uniformly hyperbolic have strong dynamical complexity in the sense of topological entropy.

In present paper, we mainly consider two typical systems beyond uniform hyperbolicity: homoclinic classes and hyperbolic ergodic measures.
On one hand, it is known that nontrival homoclinic class exists widely in dynamics beyond uniform hyperbolcity. For example, any non-trivial isolated transitive set of $C^{1}$ generic diffeomorphism is a non-trivial homoclinic class \cite{BD1999},  every transitive set of $C^{1}$ generic diffeomorphism containing a hyperbolic periodic orbit $p$ is contained in the homoclinic class of $p$ \cite{Arna2001}.
On  the other hand, every smooth compact Riemannian manifold of dimension at least 2 admits a hyperbolic ergodic measure \cite{DolPes2002}.
Let $f$ be a $C^{1}$ diffeomorphisms on $M,$ we recall that an ergodic $f$-invariant Borel probability measure is said to be hyperbolic if it has positive and negative but no zero Lyapunov exponents, the homoclinic class of a hyperbolic saddle $p$, denoted by $H(p),$ is the closure of the set of hyperbolic saddles $q$ homoclinically related to $p$ (the stable manifold of the orbit of $q$ transversely meets the unstable one of the orbit of $p$ and vice versa).

Before giving our first result,  we need some definitions and notations.
For any $x \in M$, the orbit of $x$ is $\{f^n(x)\}_{n=0}^\infty$,   denoted by $orb(x,f)$. The $\omega$-limit set of $x$ is the set of all { accumulation} points of $orb(x,f)$,   denoted by ${ \omega(f,x)}$.
A point $x \in M$ is recurrent, if $x \in { \omega(f,x)}$. We denote the sets of all recurrent points and transitive points by $Rec$ and $Trans$ respectively.
Given $x\in M$, denote $V_{f}(x)$ the set of all accumulation points of the empirical measures
$$
\mathcal{E}_n (x):=\frac1{n}\sum_{i=0}^{n-1}\delta_{f^{i} (x)},
$$
where $\delta_x$ is the Dirac measure concentrating on $x$. For a  probability measure $\mu,$  we denote it's support by $$S_\mu:=\{x\in M:\mu(U)>0\ \text{for any neighborhood}\ U\ \text{of}\ x\}.$$   We say an invariant measure $\mu$ is uniformly hyperbolic if $S_{\mu}$ is a uniformly hyperbolic set, otherwise, we say $\mu$ is not uniformly hyperbolic.
Denote $QR(f)=M\setminus IR(f).$ Let $\operatorname{Diff}^{1}(M)$ be the space of $C^{1}$ diffeomorphisms on $M.$
For a subset $Z$ of $M,$ we denote the topological entropy of $f$ on $Z$ by $h_{top}(f,Z)$ (see definition in section \ref{section-preliminaries}).
Let $\operatorname{Diff}^{1}(M)$ denote the space of $C^{1}$ diffeomorphisms on $M.$ Now we state our first result as follows.

\begin{maintheorem}\label{maintheorem-1'}
	Let $f \in \operatorname{Diff}^{1}(M)$ and $H(p)$ be a nontrival homoclinic class that is not uniformly hyperbolic. If $\Lambda\subset H(p)$ is a transitive locally maximal hyperbolic set and there is a periodic point $q\in\Lambda$ such that $q$ is homoclinically related to $p,$ then
	\begin{description}
		\item[(a)] $h_{top}(f,\{x\in M: x\text{ is not uniformly hyperbolic, }\text{each }\mu\in V_f(x) \text{ is uniformly hyperbolic}\}\}\cap IR(f)\cap Rec)\geq h_{top}(f,\Lambda)>0,$
		\item[(b)] $h_{top}(f,\{x\in M: x\text{ is not uniformly hyperbolic, }V_f(x) \text{ consists of one uniformly hyperbolic me-}\\ \text{asure}\}\cap QR(f)\cap Rec)\geq h_{top}(f,\Lambda)>0.$
	\end{description}
	Moreover, if $H(p)$ has uniform separation property, then
	\begin{description}
		\item[(c)] $h_{top}(f,\{x\in M: x\text{ is not uniformly hyperbolic, }\exists\ \mu_1,\mu_{2}\in V_f(x)\text{ s.t. } \mu_1\text{ is uniformly hyperbolic,}\\ \mu_2\text{ is not uniformly hyperbolic}\}\cap IR(f)\cap Rec)\geq h_{top}(f,\Lambda)>0,$
		\item[(d)] $h_{top}(f,\{x\in M: x\text{ is not uniformly hyperbolic, }\text{each }\mu\in V_f(x) \text{ is not uniformly hyperbolic}\}\cap IR(f)\cap Rec)\geq h_{top}(f,\Lambda)>0,$
		\item[(e)] $h_{top}(f,\{x\in M: x\text{ is not uniformly hyperbolic, }V_f(x) \text{ consists of one measure which is not uni-}\\ \text{formly hyperbolic }\}\cap QR(f)\cap Rec)\geq h_{top}(f,\Lambda)>0.$
	\end{description}
	In particular, if further $H(p)=M,$ then $Rec$ can be replaced by $Trans$ in above sets.
\end{maintheorem}

There are many transitive systems for which the whole space is a homoclinic class including \\
(i) the nonuniformly hyperbolic diffeomorphisms constructed by Katok \cite{Katok-ex}. For arbitrary compact connected two-dimensional manifold $M$, A. Katok proved that there exists a $C^\infty$ diffeomorphism $f$ such that the Riemannian volume $m$ is an $f$-invariant ergodic hyperbolic measure. From \cite{Katok} (or Theorem S.5.3 on Page 694 of book \cite{KatHas}) we know that the support of any ergodic and non-atomic hyperbolic measure of a $C^{1+\alpha}$ diffeomorphism is contained in a non-trivial homoclinic class, then there is a hyperbolic periodic point $p$ such that $M=S_m=H(p).$ Moreover, J. Buzzi \cite{Buzzi1997} showed that every $C^\infty$ diffeomorphism is asymptotically entropy expansive. Since  symptotically entropy expansive implies uniform separation property by \cite[Theorem 3.1]{PS}, then the $C^\infty$ diffeomorphism $f$ has uniform separation property.
\\(ii) generic systems in the space of robustly transitive diffeomorphisms $\operatorname{Diff}^{1}_{RT}(M).$ By the robustly transitive partially hyperbolic diffeomorphisms constructed by Ma\~{n}\'{e} \cite{Mane-ex} and the robustly transitive nonpartially hyperbolic diffeomorphisms constructed by Bonatti and Viana \cite{BV-ex}, we know that $\operatorname{Diff}^{1}_{RT}(M)$ is a non-empty open set in $\operatorname{Diff}^{1}(M).$ Since any non-trivial isolated transitive set of $C^{1}$ generic diffeomorphism is a non-trivial homoclinic class \cite{BD1999},  we have that $$\mathcal{R}_1=\{f\in \operatorname{Diff}^{1}_{RT}(M): \text{ there is a hyperbolic periodic point }  p \text{ such that } M=H(p) \}$$ is generic in  $\operatorname{Diff}^{1}_{RT}(M).$ Moreover,  $C^1$ generically in any dimension, isolated homoclinic classes are entropy expansive \cite{PV2008}.  Since  entropy expansive implies uniform separation property by \cite[Theorem 3.1]{PS}, then we have that $$\mathcal{R}_2=\{f\in \mathcal{R}_1: (M,f)\text{ has uniform separation property} \}$$ is generic in  $\operatorname{Diff}^{1}_{RT}(M).$\\
(iii) generic systems in the space of volume-preserving diffeomorphisms $\operatorname{Diff}^{1}_{vol}(M).$ Let $M$ be a compact connected Riemannian manifold. Bonatti and Crovisier proved in \cite[Theorem 1.3]{BC2004} that there exists a residual $C^1$-subset $\mathcal{R}_1$ of $\operatorname{Diff}^{1}_{vol}(M)$ such that if $f\in\mathcal{R}_1$ then $f$ is a transitive diffeomorphism. Moreover, by its proof on page 79 and page 87 of \cite{BC2004}, if $f\in\mathcal{R}_1$ then there is a hyperbolic periodic point $p$ such that $M=H(p).$ Since the space of diffeomorphisms away from homoclinic tangencies  $\operatorname{Diff}^{1}(M)\setminus\overline{HT}$ is open in $\operatorname{Diff}^{1}(M),$ then
$\mathcal{R}_2=\mathcal{R}_1\cap \operatorname{Diff}^{1}(M)\setminus\overline{HT}$ is generic in $\operatorname{Diff}^{1}_{vol}(M)\setminus\overline{HT}.$ Moreover, every $C^1$ diffeomorphism away from homoclinic tangencies is entropy expansive \cite{LVY2013}.  Note that  entropy expansive implies uniform separation property by \cite[Theorem 3.1]{PS}, if $f\in\mathcal{R}_2$ then  there is a hyperbolic periodic point $p$ such that $M=H(p)$ and $(M,f)$ has uniform separation property.

From \cite{Katok} (or Theorem S.5.3 on Page 694 of book \cite{KatHas}) we know that the support of any ergodic and non-atomic hyperbolic measure of a $C^{1+\alpha}$ diffeomorphism is contained in a non-trivial homoclinic class. In fact, for a hyperbolic ergodic measure $\mu$ and any $\eta>0,$ there exist a hyperbolic periodic point $p$ and a transitive locally maximal hyperbolic set $\Lambda$ such that $S_\mu\subset H(p),$ $p\in \Lambda \subset H(p)$ and $h_{top}(f,\Lambda)\geq h_{\mu}(f)-\eta,$ where $h_{\mu}(f)$ is the metric entropy of $\mu.$ Then we have second result.
\begin{maintheorem}\label{maintheorem-1}
	Let $f \in \operatorname{Diff}^{1}(M).$ If $f$ is a $C^{1+\alpha}$ diffeomorphism preserving a hyperbolic ergodic measure $\mu$ whose support is not uniformly hyperbolic, then
	\begin{description}
		\item[(a)] $h_{top}(f,\{x\in M: x\text{ is not uniformly hyperbolic, }\text{each }\mu\in V_f(x) \text{ is uniformly hyperbolic}\}\}\cap IR(f)\cap Rec)\geq h_{\mu}(f),$
		\item[(b)] $h_{top}(f,\{x\in M: x\text{ is not uniformly hyperbolic, }V_f(x) \text{ consists of one uniformly hyperbolic me-}\\ \text{asure}\}\cap QR(f)\cap Rec)\geq h_{\mu}(f).$
	\end{description}
	Moreover, if $(M,f)$ has uniform separation property, then
	\begin{description}
		\item[(c)] $h_{top}(f,\{x\in M: x\text{ is not uniformly hyperbolic, }\exists\ \mu_1,\mu_{2}\in V_f(x)\text{ s.t. } \mu_1\text{ is uniformly hyperbolic,}\\ \mu_2\text{ is not uniformly hyperbolic}\}\cap IR(f)\cap Rec)\geq h_{\mu}(f),$
		\item[(d)] $h_{top}(f,\{x\in M: x\text{ is not uniformly hyperbolic, }\text{each }\mu\in V_f(x) \text{ is not uniformly hyperbolic}\}\cap IR(f)\cap Rec)\geq h_{\mu}(f),$
		\item[(e)] $h_{top}(f,\{x\in M: x\text{ is not uniformly hyperbolic, }V_f(x) \text{ consists of one measure which is not uni-}\\ \text{formly hyperbolic }\}\cap QR(f)\cap Rec)\geq h_{\mu}(f).$
	\end{description}
	In particular, if further $S_{\mu}=M,$ then $Rec$ can be replaced by $Trans$ in above sets.
\end{maintheorem}

For diffeomorphism $f$ satisfying that  $$h_{top}(f)=\sup\{h_\mu(f):\mu\text{ is hyperbolic and } S_\mu \text{ is not uniformly hyperbolic}\},$$ we can replace the metric entropy of Theorem \ref{maintheorem-1} by topological entropy of $f$ as following.
\begin{maincorollary}\label{Coro-A}
	Let $f \in \operatorname{Diff}^{1}(M).$ If $f$ is a $C^{1+\alpha}$ diffeomorphism with postive topological entropy, and there is a hyperbolic measure of maximal entropy whose support is not uniformly hyperbolic, or there are a sequence of hyperbolic measures whose metric entropies converge to the topological entropy of $f$ and whose supports are not uniformly hyperbolic,
    then
    \begin{description}
    	\item[(a)] $h_{top}(f,\{x\in M: x\text{ is not uniformly hyperbolic, }\text{each }\mu\in V_f(x) \text{ is uniformly hyperbolic}\}\}\cap IR(f)\cap Rec)=h_{top}(f)>0,$
    	\item[(b)] $h_{top}(f,\{x\in M: x\text{ is not uniformly hyperbolic, }V_f(x) \text{ consists of one uniformly hyperbolic me-}\\ \text{asure}\}\cap QR(f)\cap Rec)=h_{top}(f)>0.$
    \end{description}
    Moreover, if $(M,f)$ has uniform separation property, then
    \begin{description}
    	\item[(c)] $h_{top}(f,\{x\in M: x\text{ is not uniformly hyperbolic, }\exists\ \mu_1,\mu_{2}\in V_f(x)\text{ s.t. } \mu_1\text{ is uniformly hyperbolic,}\\ \mu_2\text{ is not uniformly hyperbolic}\}\cap IR(f)\cap Rec)=h_{top}(f)>0,$
    	\item[(d)] $h_{top}(f,\{x\in M: x\text{ is not uniformly hyperbolic, }\text{each }\mu\in V_f(x) \text{ is not uniformly hyperbolic}\}\cap IR(f)\cap Rec)=h_{top}(f)>0,$
    	\item[(e)] $h_{top}(f,\{x\in M: x\text{ is not uniformly hyperbolic, }V_f(x) \text{ consists of one measure which is not uni-}\\ \text{formly hyperbolic }\}\cap QR(f)\cap Rec)=h_{top}(f)>0.$
    \end{description}
    In particular, if further $S_{\mu}=M,$ then $Rec$ can be replaced by $Trans$ in above sets.
\end{maincorollary}
\begin{remark}
	For the nonuniformly hyperbolic diffeomorphisms constructed by Katok \cite{Katok-ex}, the robustly transitive partially hyperbolic diffeomorphisms constructed by Ma\~{n}\'{e} \cite{Mane-ex} and the robustly transitive nonpartially hyperbolic diffeomorphisms constructed by Bonatti and Viana \cite{BV-ex},  note that there exists a unique measure of maximal entropy with full suport $($for example see \cite{BFSV2012}$).$ So Corollary \ref{Coro-A} holds for the three systems.
\end{remark}

\textbf{Organization of this paper.} In section \ref{section-preliminaries} we recall some definitions and lemmas used in our main
results. In section \ref{section-3} we study topological entropy of saturated sets for systems with a sequence of nondecreasing invariant compact subsets such that every subsystem has specification property, or has shadowing property and transitivity. 
In section \ref{section-4} we introduct some abtract results for possibly more applications by using the results of saturated sets obtained in section \ref{section-3}, and we give the proofs of Theorem \ref{maintheorem-1'} and \ref{maintheorem-1}.

\section{Preliminaries}\label{section-preliminaries}
Let $(X,d)$ be a nondegenerate $($i.e,
with at least two points$)$ compact metric space, and $f:X \rightarrow X$ be a continuous map. Such $(X,f)$ is called a dynamical system. For a dynamical system $(X,f)$, let $\mathcal{M}(X)$, $\mathcal{M}_{f}(X)$, $\mathcal{M}^{e}_{f}(X)$ denote the space of probability measures, $f$-invariant, $f$-ergodic probability measures, respectively.
Let $\mathbb{N}$, $\mathbb{N^{+}}$ denote non-negative integers, positive integers, respectively.

\subsection{Entropy}
Now let us to recall the definition of topological entropy in \cite{Bowen1973} by Bowen.
For $Z\subseteq X$, $s\geq 0$, $N\in\mathbb{N}^+$, and $\varepsilon >0$, define
$$\mathcal{M}_{N,\varepsilon}^{s}(Z)=\inf \sum_{i}\exp(-sn_{i})$$
where the infimum is taken over all finite or countable families $\{B_{n_{i}}(x_{i},\varepsilon)\}$ such that $x_{i}\in X$, $n_{i}\geq N$ and $\bigcup_{i}B_{n_{i}}(x_{i},\varepsilon)\supseteq Z$, where
$$
B_{n}(x,\varepsilon)=\{y\in X:d_n(x,y)<\varepsilon\},
$$
and
$$d_n(x,y)=\max\{d(f^{i}(x),f^{i}(y)):0\leq i\leq n-1\}.$$
The quantity $\mathcal{M}_{N,\varepsilon}^{s}(Z)$
does not decrease as $N$ increases and $\varepsilon$
decreases, hence the following limits exist:
$$\mathcal{M}_{\varepsilon}^{s}(Z)=\lim_{N\to \infty}\mathcal{M}_{N,\varepsilon}^{s}(Z)$$
and
$$\mathcal{M}^{s}(Z)=\lim_{\varepsilon\to 0}\mathcal{M}_{\varepsilon}^{s}(Z).$$
The Bowen topological entropy $h_{top}(f,Z)$ is defined as
a critical value of the
parameter $s$, where $\mathcal{M}^{s}(Z)$ jumps from $\infty$ to $0$, i.e.
\begin{equation*}
	\mathcal{M}^{s}(Z)=
	\begin{cases}
		0,&s>h_{top}(f,Z),\\
		\infty,&s<h_{top}(f,Z).
	\end{cases}
\end{equation*}

Some basic properties of Bowen topological entropy are as following:
\begin{Prop}\cite[Proposition 2]{Bowen1973}\label{prop-AA}
	Suppose that $(X,f)$ is a dynamical system. Then for any $Z\subseteq X,$ one has
	\begin{description}
		\item[(1)] $h_{top}(f,f(Z))=h_{top}(f,Z).$
		\item[(2)] $h_{top}(f^k,Z)=kh_{top}(f,Z)$ for any $k\in\mathbb{N^{+}}.$
	\end{description}
\end{Prop}

Brin and Katok \cite{BK} introduced the local measure-theoretical lower entropies of $\mu$ for each $\mu\in\mathcal{M}(X)$, by

\begin{equation}
	\underline{h}_{\mu}(f)=\int\underline{h}_{\mu}(f,x)d\mu,
\end{equation}
where
\begin{equation}\label{equation-GI}
	\underline{h}_{\mu}(f,x)=\lim_{\varepsilon\to 0}\liminf_{n\to\infty}-\frac{1}{n}\log\mu(B_{n}(x,\varepsilon)).
\end{equation}
Feng and Huang in \cite{Feng-Huang} proved the following result.
\begin{Lem}\label{lemma-aa}\cite[Theorem 1.2]{Feng-Huang}
	If $Z\subset X$ is is non-empty and compact, then
	\begin{equation}
		h_{top}(f,Z)= \sup\{\underline{h}_{\mu}(f):\mu\in\mathcal{M}(X),\ \mu(Z)=1\}.
	\end{equation}
\end{Lem}

Next, we recall the definition of metric entropy.
We call $(X,  \mathcal{B},  \mu)$ a probability space if $\cB$ is a Borel $\sigma$-algebra on $X$ and $\mu$ is a probability measure on $X$.   For a finite measurable partition $\xi=\{A_1,  \cdots,  A_n\}$ of a probability space $(X,  \cB,  \mu)$,   define
$$H_\mu(\xi)=-\sum_{i=1}^n\mu(A_i)\log\mu(A_i).  $$
Let $f:X\to X$ be a continuous map preserving $\mu$.   We denote by $\bigvee_{i=0}^{n-1}f^{-i}\xi$ the partition whose element is the set $\bigcap_{i=0}^{n-1}f^{-i}A_{j_i},  1\leq j_i\leq n$.   Then the following limit exists:
$$h_\mu(f,  \xi)=\lim_{n\to\infty}\frac1n H_\mu\left(\bigvee_{i=0}^{n-1}f^{-i}\xi\right)$$
and we define the metric entropy of $\mu$ as
$$h_{\mu}(f):=\sup\{h_\mu(f,  \xi):\xi~\textrm{is a finite measurable partition of X}\}.  $$

\subsection{Specification property and shadowing property.}
First, we recall the definition of specification property.
\begin{Def}\label{definition of specification}
	We say a dynamical system $(X,f)$ has specification property if, for
	any $\varepsilon > 0$, there is a positive integer $K_\varepsilon$ such that for any integer $s\ge$ 2, any set $\{x_1,x_2,\cdots,x_s\}$ of $s$ points of $X$, and any sequence\\
	$$0 = a_1 \le b_1 < a_2 \le b_2 < \cdots < a_s \le b_s$$ of 2$s$ non-negative integers with $a_{m+1}-b_m \ge K_\varepsilon$ for $m = 1,2,\cdots,s-1$, there is a point $z$ in $X$ such that
	$$d(f^j(z),f^{j-a_i}(x_i))\leq\varepsilon \text{ for any } a_i\leq j\leq b_i\text{ and }1\leq i\leq s.$$
	In other words, $\cap_{i=1}^{s}f^{-a_i}\overline{B}_{b_i-a_i}(x_i,\varepsilon)\neq\emptyset,$ where $\overline{B}_{n}(x,\varepsilon)=\{y\in X:d_n(x,y)\leq\varepsilon\}.$
\end{Def}

Now, we recall the definition of shadowing property.
\begin{Def}
	An infinite sequence $(x_{n})_{n=0}^{\infty}$ of points
	in $X$ is a $\delta$-pseudo-orbit for a dynamical system $(X,f)$ if $d(x_{n+1},f(x_{n}))<\delta$ for each $n \in \mathbb{N}$. We say that a dynamical system $(X,f)$ has the shadowing property if for every $\varepsilon>0$ there is a $\delta >0$ such that any $\delta$-pseudo-orbit $(x_{n})_{n=0}^{\infty}$ can be $\varepsilon$-traced by a point
	$y\in X$, that is $d(f^{n}(y),x_{n})<\varepsilon$ for all $n \in \mathbb{N}$.
\end{Def}

Recall that $(X,f)$ is mixing if for any non-empty open sets $U,V\subseteq X,$ there is $N\in\mathbb{N}$ such that $f^{-n}U\cap V\neq\emptyset$ for any $n\geq N.$
\begin{Prop}\label{prop-specif}\cite[Proposition 23.20]{Sig}
	Suppose that a dynamical system $(X,f)$ has shadowing property and is mixing, then $(X,f)$ has specification property.
\end{Prop}

\begin{Def}\label{definition of entropy-dense property}
	We say $f$ satisfies the entropy-dense property if for any $\mu \in \mathcal{M}_{f}(X)$, for any neighbourhood $G$ of $\mu$ in $\mathcal{M}(X)$, and for any $\eta >0$, there exists a closed $f$-invariant set $\Lambda_{\mu}\subseteq X$, such that  $\mathcal{M}_{f}(\Lambda_{\mu})\subseteq G$ and $h_{top}(f,\Lambda_{\mu})>h_{\mu}-\eta$. By the classical variational principle, it is equivalent that for any neighbourhood $G$ of $\mu$ in $\mathcal{M}(X)$, and for any $\eta >0$, there exists a $\nu \in \mathcal{M}_{f}^{e}(X)$ such that $h_{\nu}>h_{\mu}-\eta$ and $\mathcal{M}_{f}(S_{\nu})\subseteq G$.
\end{Def}
By \cite[Proposition 2.3(1)]{PS}, entropy-dense property holds for systems with approximate product property. From definitions, if a dynamical system has specification property or has shadowing property and transitivity, then it has  approximate product property. So we have the following.

\begin{Prop}\cite[Proposition 2.3 (1)]{PS}\label{proposition of entropy-dense property}
	Suppose that $(X,f)$ is a dynamical system satisfying one of the following:
	\begin{description}
		\item[(1)] specification property;
		\item[(2)] shadowing property and transitivity.
	\end{description}
	Then $(X,f)$ has entropy-dense property.
\end{Prop}

\subsection{Uniform separation property}
For $\delta >0$ and $\varepsilon >0$, two points $x$ and $y$ are $(\delta,n,\varepsilon)$-separated if
\begin{equation*}
	|\{0\leq j\leq n-1:d(f^{j}(x),f^{j}(y))>\varepsilon\}| \geq \delta n.
\end{equation*}
A subset $E$ is $(\delta,n,\varepsilon)$-separated  if any pair of different points of $E$ are $(\delta,n,\varepsilon)$-separated. Let $F\subseteq \cM(X)$  be a neighbourhood of $\nu \in \cM_{f}(X)$. Define $X_{n,F}:=\{x\in X:\mathcal{E}_{n}(x)\in F\}.$

We need to use uniform separation property in the following form  which is slightly different from the original definition given by Pister and Sullivan in \cite{PS}.
\begin{Def}\label{def-uniform-separation}
	We say that the dynamical system $(X,f)$ satisfies uniform separation property, if the following holds. For any $\eta >0$, there exists $\delta^{*}>0$, $\varepsilon^{*}>0$ such that for $\mu$  ergodic and
	any neighbourhood $F\subseteq \cM(X)$ of $\mu$, there exists $n_{F,\mu,\eta}^{*}$,  such that for $n\geq n_{F,\mu,\eta}^{*}$, there is a $(\delta^{*},n,\varepsilon^{*})$-separated set $\Gamma_n\subseteq X_{n,  F}\cap S_{\mu}$ with
	\begin{equation*}
		|\Gamma_n|\geq e^{n(h_{\mu}(f)-\eta)}.
	\end{equation*}
\end{Def}
The only difference between Definition \ref{def-uniform-separation} and the original definition is that $\Gamma_n\subseteq X_{n,  F}$ in the latter.  In \cite[Theorem 3.1]{PS}, it's proved that if a dynamical system is expansive, or is asymptotically $h$-expansive, then it has the uniform separation property of the original form. By \cite[Theorem 3.1]{PS}, we have the following.
\begin{Prop}\label{proposition-AA}
	If $(X,  f)$ is expansive or asymptotically $h$-expansive, then $(X,f)$ satisfies uniform separation property.
\end{Prop}
\begin{proof}
	By \cite[Equation (14)]{PS} and the definition $\Xi'$, for any $\eta >0$, there exists $\delta^{*}>0$, $\varepsilon^{*}>0$ such that for $\mu$  ergodic and
	any neighbourhood $F\subseteq \cM(X)$ of $\mu$, there exists $n^{*}$ such that for $n\geq n_{F,\mu,\eta}^{*},$  there exist disjoint sets $\{B_i\}_{i=1}^{k}$ with $k=|\Xi'|$ such that $B_i\subset X_{n,F},$ and $\mu(B_i)>0$ for any $1 \leq i\leq k,$ and for any $x_i\in B_i,$ $x_j\in B_j$ with $1\leq i\neq j\leq k,$ $x_i$ and $x_j$ are $(\delta^{*},n,\varepsilon^{*})$-separated.
	
	Let $C_i=B_i\cap S_{\mu}$ for any $1 \leq i\leq k.$ Then $C_i\subset X_{n,F}\cap S_\mu,$ and $\mu(C_i)>0$ for any $1 \leq i\leq k,$ and for any $x_i\in C_i,$ $x_j\in C_j$ with $1\leq i\neq j\leq k,$ $x_i$ and $x_j$ are $(\delta^{*},n,\varepsilon^{*})$-separated. Take one point $y_i\in C_i$ for any $1\leq i\leq k,$ and let $\Gamma_n=\{y_i\}_{i=1}^{k}.$ Then $\Gamma_n\subseteq X_{n,  F}\cap S_{\mu}$ and
	$
		|\Gamma_n|=|\Xi'|\geq e^{n(h_{\mu}(f)-\eta)}.
	$
\end{proof}

For $\eps>0$ and $n\in\N$,   two points $x$ and $y$ are $(n,  \eps)$-separated if
$d_n(x,y)>\eps.$
A subset $E$ is $(n,  \eps)$-separated if any pair of different points of $E$ are $(n,  \eps)$-separated.
For a fixed $\delta>0$,   when $n$ is large enough,   a $(\delta,  n,  \eps)$-separated set is also an $(n,  \eps)$-separated set.   So we have the following lemma.
\begin{Lem}\label{lem-n-eps}
	If $(X,  f)$ has the uniform separation property,   then for any $\eta>0$,   there exists $\widetilde{\eps}>0$ so that for $\mu$ ergodic and any neighbourhood $F\subseteq \cM(X)$ of $\mu$,   there exists $\widetilde{n}_{F,  \mu,  \eta}\in\N$ such that for any $n\geq \widetilde{n}_{F,  \mu,  \eta}$,   there is a $(n,  \widetilde{\eps})$-separated set $\Gamma_n\subseteq X_{n,  F}\cap S_{\mu}$ with
	$$|\Gamma_n|\geq e^{n(h_{\mu}(f)-\eta)}. $$
\end{Lem}

\section{Saturated sets}\label{section-3}
One main technique in the proofs of Theorem \ref{maintheorem-1'} and \ref{maintheorem-1} is using the saturated sets which can avoid a long construction proof for every object being considered.
We say a dynamical system $(X,f)$ have {saturated property}, if for any compact connected nonempty set $K \subseteq \mathcal M_f(X),$
\begin{eqnarray} \label{eq- saturated-definition}
	h_{top} (f,G_K )=\inf\{h_\mu (f):\mu\in K\},
\end{eqnarray}  
where $G_{K}
=\{x\in X: V_f(x)=K\} $ (called saturated set). 
Note that for any $x\in X$, $V_{f}(x)$ is always a nonempty compact connected subset of $\cM_f(X)$ by \cite[Proposition 3.8]{Sig}, so $G_{K}\neq\emptyset$ requires that $K$ is a nonempty compact connected set. 
The existence of saturated sets is proved by Sigmund \cite{SigSpe} for systems with uniform hyperbolicity or specification property and generalized to  non-uniformly hyperbolic systems in \cite{LST}. 
The property on entropy estimate was firstly established by Pfister and
Sullivan in \cite{PS2}, provided that the system has $g$-product property (which is weaker than specification property) and uniform separation property (which is weaker than expansiveness). Recently, in \cite[Theorem 1.4]{HTW} the authors proved that if further there is an invariant measure with full support, then
\begin{equation}\label{equ:Tian's result}
	G_{K}\cap Trans\neq\emptyset,\text{ and }h_{top}(f,G_{K}\cap Trans)=\inf\{h_{\mu}(f):\mu\in K\}.
\end{equation}

We say that a subset $A$ of $X$ is $f$-invariant (or simply invariant) if $f(A)\subset A.$ When $f$ is a homeomorphism from $X$ onto $X$, we say that a subset $A$ of $X$ is $f$-invariant if $f(A)= A.$ If $A$ is a closed $f$-invariant subset of $X,$ then $(A,f|_A)$ also is a dynamical system. We will call it a subsystem of $(X,f).$
Now, we state a theorem which study topological entropy of saturated sets for systems which have a sequence of nondecreasing invariant compact subsets such that every subsystem has specificaton property.
\begin{maintheorem}\label{maintheorem-2}
	Suppose that $(X,f)$ is a dynamical system with a sequence of nondecreasing $f$-invariant compact subsets $\{X_{n} \subseteq X:n \in \mathbb{N^{+}} \}$ such that $\overline{\bigcup_{n\geq 1}X_{n}}=X,$ and $({X_{n}},f|_{X_{n}})$ has specification property for any $n \in \mathbb{N^{+}}.$ 
	\begin{description}
		\item[(1)] If $(X,f)$ has uniform separation property, then for any non-empty compact connected subset $K\subseteq \{\mu \in \mathcal{M}_{f}(X):\mu(\bigcup_{n\geq 1}X_{n})=1\}$ and any non-empty open set $U\subseteq X,$ we have 
		$$G_{K}\cap U\cap Trans\neq\emptyset,\text{ and }h_{top}(f,G_{K}\cap U\cap Trans\})=\inf\{h_{\mu}(f):\mu\in K\}.$$
		\item[(2)]  If $(X_{n_0},f|_{X_{n_0}})$ has uniform separation property for some ${n_0}\in\mathbb{N^{+}}$, then for any non-empty compact connected subset $K\subseteq \cM_{f|_{X_{n_0}}}(X_{n_0})$ and any non-empty open set $U\subseteq X,$ we have 
		$$G_{K}\cap U\cap Trans\neq\emptyset,\text{ and }h_{top}(f,G_{K}\cap U\cap Trans\})=\inf\{h_{\mu}(f):\mu\in K\}.$$
	\end{description} 
\end{maintheorem}

If we replace specification property by shadowing property and transitivity in Theorem \ref{maintheorem-2}, we have the same results as follows.
\begin{maintheorem}\label{maintheorem-3}
	Suppose that $f$ is a homeomorphism from $X$ onto $X$, $(X,f)$ has a sequence of nondecreasing $f$-invariant compact subsets $\{X_{n} \subseteq X:n \in \mathbb{N^{+}} \}$ such that $\overline{\bigcup_{n\geq 1}X_{n}}=X$, $\mathrm{Per}(f|_{X_{1}})$ $\neq \emptyset,$ $({X_{n}},f|_{X_{n}})$ has shadowing property and is transitive for any $n \in \mathbb{N^{+}}.$ 
	\begin{description}
	\item[(1)] If $(X,f)$ has uniform separation property, then for any non-empty compact connected subset $K\subseteq \{\mu \in \mathcal{M}_{f}(X):\mu(\bigcup_{n\geq 1}X_{n})=1\}$ and any non-empty open set $U\subseteq X,$ we have 
	$$G_{K}\cap U\cap Trans\neq\emptyset,\text{ and }h_{top}(f,G_{K}\cap U\cap Trans\})=\inf\{h_{\mu}(f):\mu\in K\}.$$
	\item[(2)]  If $(X_{n_0},f|_{X_{n_0}})$ has uniform separation property for some ${n_0}\in\mathbb{N^{+}}$, then for any non-empty compact connected subset $K\subseteq \cM_{f|_{X_{n_0}}}(X_{n_0})$ and any non-empty open set $U\subseteq X,$ we have 
	$$G_{K}\cap U\cap Trans\neq\emptyset,\text{ and }h_{top}(f,G_{K}\cap U\cap Trans\})=\inf\{h_{\mu}(f):\mu\in K\}.$$
\end{description} 
\end{maintheorem}

\subsection{ Proof of Theorem \ref{maintheorem-2}(1)} 
Before proof, we introduce some basic facts and lemmas.
\subsubsection{Some lemmas}
If $r,s\in\mathbb{N},r\leq s$, we set $[r,s]:=\{j\in\mathbb{N}:\ r\leq j\leq s\}$, and the cardinality of a finite set $\Lambda$ is denoted by $|\Lambda|$. We set
$$
\langle \varphi,\mu \rangle\ :=\ \int_X\varphi d\mu.
$$
There exists a countable and separating set of continuous functions $\{\varphi_1,\varphi_2,\cdots\}$ with $0\leq \varphi_k(x)\leq 1$, and such that
{ $$
	d(\mu,\nu)\ :=\ \sum_{k\geq 1}2^{-k}\mid\langle \varphi_k,\mu\rangle-\langle \varphi_k,\nu \rangle\mid
	$$}
defines a metric for the weak*-topology on $ \mathcal M_f(X)$. We refer to \cite{PS2} and use the metric on $X$ as following defined by Pfister and Sullivan.
$$
d(x,y) := d(\delta_x,\delta_y),
$$
which is equivalent to the original metric on $X$. Readers will find the benefits of using this metric in our proof later.
Denote an open ball in $\cM(X)$ by
$$
\mathcal{B}(\nu, \zeta):=\{\mu \in \cM(X): d(\nu, \mu)< \zeta\},
$$
an open ball in $X$ by
$$
B(x, \eps):=\{y\in X: d(y,x) < \eps\}.
$$

\begin{Lem}\label{measure distance}
For any $\varepsilon > 0,\delta >0$, and any two sequences $\{x_i\}_{i=0}^{n-1},\{y_i\}_{i=0}^{n-1}$ of $X$, if $d(x_i,y_i)<\varepsilon$ holds for any $i\in [0,n-1]$, then for any $J\subseteq \{0,1,\cdots,n-1\}$ with $\frac{n-|J|}{n}<\delta$, one has:
\begin{description}
\item[(a)] $d(\frac{1}{n}\sum_{i=0}^{n-1}\delta_{x_i},\frac{1}{n}\sum_{i=0}^{n-1}\delta_{y_i})<\varepsilon.$
\item[(b)] $d(\frac{1}{n}\sum_{i=0}^{n-1}\delta_{x_i},\frac{1}{|J|}\sum_{i\in J}\delta_{y_i})<\varepsilon+2\delta.$
\end{description}
\end{Lem}

Lemma \ref{measure distance} is easy to be verified and shows us that if any two orbit of $x$ and $y$ in finite steps are close in the most time, then the two empirical measures induced by $x,y$ are also close.

\begin{Lem}\label{lemma-separation}
	For any $\eps_1 > 3\eps_2 >0$, and any integers $n,m\in\mathbb{N}$, if $x_1$ and $x_2$ are $(n,  \eps_1)$-separated, $f^m(y_1)\in B_n(x_1,\eps_2),$ $f^m(y_2)\in B_n(x_2,\eps_2),$ then $y_1$ and $y_2$ are $(n,  \eps_2)$-separated.
\end{Lem}
Lemma \ref{lemma-separation} is easy to be verified.

\begin{Lem}\label{lemma-MM}
	Suppose that $(X,f)$ is a dynamical system with a sequence of nondecreasing $f$-invariant compact subsets $\{X_{n} \subseteq X:n \in \mathbb{N^{+}} \}$ such that $\overline{\bigcup_{n\geq 1}X_{n}}=X.$ 
	For any $\mu \in \mu \in \mathcal{M}_{f}(X)$ with $\mu(\bigcup_{n\geq 1}X_{n})=1,$ let $\mu_{n}:= \frac{\mu}{\mu(X_{n})}$, then we have $\mu_{n} \in  \mathcal{M}_{f|_{X_n}}(X_n)$ and  
	$\lim\limits_{n\to\infty}\mu_{n}=\mu,$  $\lim\limits_{n\to\infty}h_{\mu_{n}}(f)=h_{\mu}(f).$ 
\end{Lem}
\begin{proof}
	For any continuous function $\varphi,$ we have
	\begin{equation*}
	\begin{split}
	& |\int_{X}\varphi d\mu-\int_{X}\varphi d\mu_{n}|\\
	= & |\int_{X_{n}^{c}}\varphi d\mu+\int_{X_{n}}\varphi d\mu-\int_{X_{n}}\varphi d\mu_{n}|\\
	\le & ||\varphi ||\mu(X_{n}^{c})+|\int_{X_{n}}\varphi d\mu-\frac{1}{\mu(X_{n})}\int_{X_{n}}\varphi d\mu|\\
	\le & (1-\mu(X_{n}))||\varphi ||+(\frac{1}{\mu(X_{n})}-1)||\varphi ||\\
	= & (\frac{1}{\mu(X_{n})}-\mu(X_{n}))||\varphi ||.
	\end{split}
	\end{equation*}
	Since $\lim\limits_{n \to \infty}\mu(X_{n})=1,$ then $\lim\limits_{n\to\infty}\mu_{n}=\mu$ by the weak*-topology of $\mathcal{M}(X)$. Let $\nu_{n}:= \frac{\mu}{\mu(X_{n}^c)}$, then $\mu=\mu(X_{n})\mu_n+(1-\mu(X_{n}))\nu_{n}.$ So $\lim\limits_{n\to\infty}h_{\mu_{n}}(f)=h_{\mu}(f)$ by $h_{\mu}(f)=\mu(X_{n})h_{\mu_n}(f)+(1-\mu(X_{n}))h_{\nu_{n}}(f).$  
\end{proof}

The following lemma is from Bowen \cite{Bowen1973}.
\begin{Lem}\label{lem-Bowen}
	Suppose that $(X,f)$ is a dynamical system.   Set
	$$QR(t)=\{x\in X:\text{ there is }\mu\in V_{f}(x)\text{ such that }h_{\mu}(f)\leq t\},$$
	then
	$h_{top}(f,  QR(t))\leq t.  $
\end{Lem}

\begin{Lem}\cite[Page 944]{PS2}\label{lemma-A}
	Suppose that $(X,f)$ is a dynamical system. If $K\subseteq \cM_f(X)$ is a non-empty compact connected set, then there exists a sequence $\left\{\alpha_{1}, \alpha_{2}, \cdots\right\}$ in $K$ such that
	$$
	\overline{\left\{\alpha_{j}: j \in \mathbb{N}^+, j>n\right\}}=K, \forall n \in \mathbb{N}^+ \text { and } \lim _{j \rightarrow \infty} d\left(\alpha_{j}, \alpha_{j+1}\right)=0.
	$$
\end{Lem}

\subsubsection{Proof of Theorem \ref{maintheorem-2}(1)}
Since $G_{K}\cap U\cap Trans\subseteq G_K$ and $h_{top}(f,G_K)\leq\inf\{h_\mu (f):\mu\in K\}$ by Lemma \ref{lem-Bowen},   one has 
\begin{equation}\label{equation-AE}
	h_{top}(f,G_{K}\cap U\cap Trans)\leq \inf\{h_\mu (f):\mu\in K\}.
\end{equation}
So it remains to show that
\begin{equation}\label{eq-G-K-U-geq-inf}
	h_{top}(f,G_{K}\cap U\cap Trans)\geq \inf\{h_\mu (f):\mu\in K\}.
\end{equation}

For any non-empty open set $U$, we can fix an $\varepsilon_0>0$ and a point $x_0\in X_{l_{0}}$ for some $l_{0} \geq 1$ such that $\overline{B}(x_0,\varepsilon_0):=\{y\in X: d(y,x_0) \leq \eps_0\}\subseteq U$. 
Let $\eta>0.$
Since $(X,f)$ has uniform separation property, by Lemma \ref{lem-n-eps} there exists $\eps^*>0$ so that for $\mu$ ergodic and any neighbourhood $F\subseteq \cM(X)$ of $\mu$,   there exists $n_{F,  \mu,  \eta}^*\in\N$ such that for any $n\geq n_{F,  \mu,  \eta}^*$,   there is a $(n,  \eps^*)$-separated set $\Gamma_n\subseteq X_{n,  F}\cap S_\mu$ with
\begin{equation}\label{equation-uniform-separation}
	|\Gamma_n|\geq e^{n(h_{\mu}(f)-\eta)}.
\end{equation}

Let $\varepsilon_1=\min\{\frac{1}{6}\eps^{*},\frac{1}{2}\varepsilon_0\}$ and  $\varepsilon_n=\varepsilon_1/2^n$ for any $n\geq 2.$ For any $k \geq 1$, $X_{k}$ is compact, thus there is a finite set $\Delta _{k}:=\{x_{1}^{k},x_{2}^{k},\dots,x_{t_{k}}^{k}\}\subseteq X_{k}$ such that $\Delta _{k}$ is $\varepsilon_{k}$-dense in $X_{k}$.

By Lemma \ref{lemma-A},  there exists a sequence $\left\{\alpha_{1}, \alpha_{2}, \cdots\right\}$ in $K$ such that
$$
\overline{\left\{\alpha_{j}: j \in \mathbb{N}^+, j>n\right\}}=K, \forall n \in \mathbb{N}^+ \text { and } \lim _{j \rightarrow \infty} d\left(\alpha_{j}, \alpha_{j+1}\right)=0.
$$
For any $k \in \mathbb{N}^+$, by Lemma \ref{lemma-MM}, there exists $\beta_k\in \cM_{f|_{X_{l_k}}}(X_{l_k})$ for some $l_k \in \mathbb{N^{+}}$ such that $d(\beta_k,\alpha_k)<\varepsilon_{k}$ and $h_{\beta_k}(f)\geq h_{\alpha_k}(f)-\eta.$
By applying Proposition \ref{proposition of entropy-dense property} to $(X_{l_k},f)$, there exists $\gamma_k\in \cM_{f|_{X_{l_k}}}^e(X_{l_k})$ such that $d(\gamma_k,\beta_k)<\varepsilon_{k}$ and $h_{\gamma_k}(f)\geq h_{\beta_k}(f)-\eta.$ Then we have
\begin{equation}\label{equation-D}
	\overline{\left\{\gamma_{j}: j \in \mathbb{N}^+, j>n\right\}}=\overline{\left\{\alpha_{j}: j \in \mathbb{N}^+, j>n\right\}}=K, \forall n \in \mathbb{N}^+\text { and } \lim _{j \rightarrow \infty} d\left(\gamma_{j}, \gamma_{j+1}\right)=0.
\end{equation}

Denote $H^{*}= \inf\{h_{\gamma_k}(f):j\in\mathbb{N^{+}}\}-\eta.$
Let $L_{1}=\max\{l_1,1,l_0\}$ and $L_{k}=\max\{l_{k},j,L_{k-1}\}$ for any $k\geq 2,$ then one has $L_{k} \leq L_{k+1}$ for any $k \geq 1$.
Let $\{\zeta_{k}\}$ be one strictly decreasing sequence so that 
\begin{equation}\label{equation-AF}
	\lim\limits_{k\rightarrow \infty}\zeta_{k}=0\text{ and }5\zeta_{1}(H^*-\eta)<\eta.
\end{equation}
From (\ref{equation-uniform-separation}), we get the existence of $n_{k}$ and a $\left(n_{k}, \varepsilon^{*}\right)$-separated subset $\Gamma_{k} \subseteq X_{n_{k}, \mathcal{B}\left(\gamma_{k}, \zeta_{k}\right)}\cap S_{\gamma_k}\subseteq X_{n_{k}, \mathcal{B}\left(\gamma_{k}, \zeta_{k}\right)}\cap X_{l_k}$ with
\begin{equation}\label{equation-AC}
	|\Gamma_{k}| \geq e^{n_{k} (h_{\gamma_k}(f)-\eta)}\geq e^{n_{k} H^{*}},
\end{equation}
where $\mathcal{B}\left(\gamma_{k}, \zeta_{k}\right)=\{\gamma \in \cM(X): d(\gamma,\gamma_{k})\leq\zeta_k\}.$
We may assume that $n_k$ satisfies
\begin{equation}\label{equation-B}
	\frac{t_kK_k+K_{k+1}}{n_k}\leq \zeta_{k},
\end{equation}
where $K_k=K_{\eps_{k}}$ is defined in Definition \ref{definition of specification} for $(X_{L_{k}},f)$ and $\eps_k.$

Next, we choose a strictly increasing $\left\{N_{k}\right\}_{k=1}^{+\infty}$ with $N_{k} \in \mathbb{N}$ such that for any $k\geq 1,$ on has
\begin{equation}\label{equation-A}
	\begin{split}
		&n_{k+1}+(t_{k+1}-1)K_{k+1}+2K_{k+2}+n_{k+2} \leq \zeta_{k} \sum_{j=1}^{k}N_{j}n_{j},\\
		&\sum_{j=1}^{k}\left(N_{j}(n_{j}+K_j-1)+t_{j}K_j\right)+K_{k+1} \leq \zeta_{k+1}\sum_{j=1}^{k+1}N_{j}n_{j}.
	\end{split}
\end{equation}
Now, we define a sequences $\{n_{j}'\}_{j=1}^{+\infty}$ as
$$
n_{j}'=\left\{\begin{array}{ll}
	K_k, & \text { for } j = 0,\\
	n_{k}+K_k-1, & \text { for } N_{1}+N_{2}+\dots+N_{k-1}+t_{1}+\dots+t_{k-1}+q\ \mathrm{with}\  1\leq q \leq N_{k},\\
	K_k, & \text { for } N_{1}+N_{2}+\dots+N_{k}+t_{1}+\dots+t_{k-1}+q \ \mathrm{with}\ 1\leq q \leq t_{k}-1,\\
	K_{k+1}, & \text { for } N_{1}+N_{2}+\dots+N_{k}+t_{1}+\dots+t_{k},
\end{array}\right.
$$
where $N_0=t_0=0.$
Let $M_{-1}=0$ and $M_{j}=\sum_{i=0}^{j}n_i'$ for any $j\in\mathbb{N}.$

Denote $\Gamma_{k}'=\Gamma_{k}\times \cdots \times\Gamma_{k}=\{(y_1^k,\dots,y_{N_k}^k):y_j^k\in\Gamma_{k}\text{ for } 1\leq j\leq N_k\}.$ Then $|\Gamma_{k}'|=|\Gamma_{k}|^{N_k}.$
Now, we inductively construct  a subset of $G_{K}\cap U\cap Trans$ such that its topological entropy is close to $\inf\{h_\mu(f):\mu\in K\}.$ This construction is the cornerstone of the proof.

{\bf Step 1: construct $G_{1}.$} 
By specification property of $(X_{L_1},f|_{X_{L_1}}),$ for any $\mathbf{y}_1=(y_1^1,\dots,y_{N_1}^1)\in \Gamma_1',$
\begin{equation*}
		G_1(\mathbf{y}_1,\eps_1):=\overline{B}(x_0,\eps_1)\bigcap \left(\cap_{j=1}^{N_1}f^{-M_{j-1}}\overline{B}_{n_1}(y_j^1,\eps_1)\right)
		\bigcap \left(\cap_{j=1}^{t_1}f^{-M_{N_1+j-1}}\overline{B}(x_j^1,\eps_1)\right)
		\neq\emptyset,
\end{equation*}
where $\overline{B}_{n}(x,\varepsilon)=\{y\in X:d_n(x,y)\leq\varepsilon\}.$
Let $\mathbf{y}_1^1=(y_1^{1,1},\dots,y_{N_1}^{1,1}),\ \mathbf{y}_1^2=(y_1^{1,2},\dots,y_{N_1}^{1,2})\in \Gamma_1'$ with $\mathbf{y}_1^1\neq \mathbf{y}_1^2,$ if $y_{j'}^{1,1}\neq y_{j'}^{1,2}$ for some $1\leq j'\leq N_1,$ then for any $z_1^1\in G_1(\mathbf{y}_1^{1},2\eps_1)$ and $z_1^2\in G_1(\mathbf{y}_1^2,2\eps_1),$ 
\begin{equation}\label{equation-AA}
	z_1^1\text{ and }z_1^2\text{ are }(M_{j'},\frac{1}{3}\eps^*)\text{-separated}
\end{equation} 
by Lemma \ref{lemma-separation}.
We choose $$z_1=z_1(\mathbf{y}_1)\in G_1(\mathbf{y}_1,\eps_1)$$ for any $\mathbf{y}_1\in \Gamma_1''$, then $$F_{1}:=\{z_1(\mathbf{y}_1):\mathbf{y}_1\in \Gamma_1''\}$$ is a $(M_{N_1+t_1-1},\frac{1}{3}\eps^*)$-separated set and $|F_1|=|\Gamma_1''|$.
Let $$G_1=\bigcup_{\mathbf{y}_1\in \Gamma_1''}G_1(\mathbf{y}_1,2\eps_1),$$ then $G_1$ is a non-empty closed sets.

{\bf Step k: construct $G_{k}.$} 
By specification property of $(X_{L_k},f|_{X_{L_k}}),$ for any $z_{k-1}=z_{k-1}(\mathbf{y}_1,\dots,\mathbf{y}_{k-1})\in F_{k-1}$ and any $\mathbf{y}_k=(y_1^k,\dots,y_{N_k}^k)\in \Gamma_k',$ 
\begin{equation*}
	\begin{split}
		G_k(\mathbf{y}_1,\dots,\mathbf{y}_k,\eps_k):=&\overline{B}_{M_{\sum_{i=1}^{k-1}(N_i+t_i)-1}}(z_{k-1},\eps_k)\bigcap \left(\cap_{j=1}^{N_k}f^{-M_{\sum_{i=1}^{k-1}(N_i+t_i)-1+j}}\overline{B}_{n_k}(y_j^k,\eps_k)\right)\\
		&\bigcap \left(\cap_{j=1}^{t_k}f^{-M_{\sum_{i=1}^{k-1}(N_i+t_i)-1+N_k+j}}\overline{B}(x_j^k,\eps_k)\right)
		\neq\emptyset.
	\end{split}
\end{equation*}
Then $\emptyset \neq G_k(\mathbf{y}_1,\dots,\mathbf{y}_k,\eps_k)\subseteq G_k(\mathbf{y}_1,\dots,\mathbf{y}_k,2\eps_k)\subseteq  \overline{B}_{M_{\sum_{i=1}^{k-1}(N_i+t_i)-1}}(z_{k-1},2\eps_k).$ By $\eps_{k-1}=2\eps_{k}$ and $z_{k-1}\in G_k(\mathbf{y}_1,\dots,\mathbf{y}_{k-1},\eps_{k-1}),$  we have
\begin{equation}\label{equation-C}
	G_k(\mathbf{y}_1,\dots,\mathbf{y}_k,2\eps_k)\subseteq G_k(\mathbf{y}_1,\dots,\mathbf{y}_{k-1},2\eps_{k-1}).
\end{equation}
Let $\mathbf{y}_i^1=(y_1^{i,1},\dots,y_{N_i}^{i,1}),\ \mathbf{y}_i^2=(y_1^{i,2},\dots,y_{N_i}^{i,2})\in \Gamma_i'$ for each $1\leq i\leq k$ with $(\mathbf{y}_1^1,\dots,\mathbf{y}_k^1)\neq (\mathbf{y}_1^2,\dots,\mathbf{y}_k^2),$ if $y_{j'}^{i',1}\neq y_{j'}^{i',2}$ for some $1\leq i'\leq k$ and some $1\leq j'\leq N_{i'},$ then for any $z_k^1\in G_k(\mathbf{y}_1^1,\dots,\mathbf{y}_k^1,2\eps_k)$ and $z_k^2\in G_k(\mathbf{y}_1^2,\dots,\mathbf{y}_k^2,2\eps_k),$
\begin{equation}\label{equation-AB}
	z_k^1\text{ and }z_k^2\text{ are }(M_{\sum_{i=1}^{i'-1}(N_i+t_i)+j'},\frac{1}{3}\eps^*)\text{-separated}
\end{equation}
by Lemma \ref{lemma-separation}, (\ref{equation-C}) and (\ref{equation-AA}).
We choose $$z_{k}=z_{k}(\mathbf{y}_1,\dots,\mathbf{y}_{k})\in G_k(\mathbf{y}_1,\dots,\mathbf{y}_k,\eps_k)$$ for any $(\mathbf{y}_1,\dots,\mathbf{y}_{k})\in  \Gamma_1'\times\dots\times\Gamma_k'$, then $$F_{k}:=\{z_k(\mathbf{y}_1,\dots,\mathbf{y}_{k}):(\mathbf{y}_1,\dots,\mathbf{y}_{k})\in  \Gamma_1'\times\dots\times\Gamma_k'\}$$ is a $(M_{\sum_{i=1}^{k}(N_i+t_i)-1},\frac{1}{3}\eps^*)$-separated set and $|F_k|=|\Gamma_1'\times\dots\times\Gamma_k'|=\prod_{i=1}^{k}|\Gamma_i'|$.
Let $$G_k=\bigcup_{(\mathbf{y}_1,\dots,\mathbf{y}_k)\in \Gamma_1'\times\dots\times\Gamma_k'}G_k(\mathbf{y}_1,\dots,\mathbf{y}_k,2\eps_k),$$ then $G_k$ is a non-empty closed sets, and $G_{k}\subset G_{k-1}$ by (\ref{equation-C}).

Let
\begin{equation*}
	G:=\bigcap_{k\geq 1}G_k.
\end{equation*}
Note that $G_{k+1}\subset G_k$ for any $k\in\mathbb{N^{+}},$ then $G$ is a non-empty closed set.

Next we will prove the following:
\begin{enumerate}
	\item $G\subseteq Trans\cap U$.
	\item $G\subseteq G_{K}$.
	\item $h_{top}(f,G)\geq \inf\{h_\mu (f):\mu\in K\}-4\eta.$
\end{enumerate}

\textbf{Proof of Item $(1)$}: 
Fix $z\in G.$ 
For any $x\in X$ and any $k\in\mathbb{N^{+}},$ there is $k'\geq k$ and $x'\in X_{k'}$ such that $d(x,x')\leq \eps_{k}$ by $\overline{\bigcup_{n\geq 1}X_{n}}=X.$ Since $\Delta _{k'}:=\{x_{1}^{k'},x_{2}^{k'},\dots,x_{t_{k'}}^{k'}\}\subseteq X_{k'}$ is $\varepsilon_{k'}$-dense in $X_{k,}$, then there is $1\leq i'\leq t_{k'}$ such that $d(x',x_{i'}^{k'})\leq \eps_{k'}.$ By  $$z\in G\subseteq G_{k'}\subseteq \cap_{j=1}^{t_{k'}}f^{-M_{\sum_{i=1}^{k'-1}(N_i+t_i)-1+N_{k'}+j}}\overline{B}(x_j^{k'},2\eps_{k'}),$$ we have $d(f^{M_{\sum_{i=1}^{k'-1}(N_i+t_i)-1+N_{k'}+i'}}(z),x_{i'}^{k'})\leq 2\eps_{k'}.$   Thus
\begin{equation*}
	\begin{split}
		d(f^{M_{\sum_{i=1}^{k'-1}(N_i+t_i)-1+N_{k'}+i'}}(z),x)&\leq d(f^{M_{\sum_{i=1}^{k'-1}(N_i+t_i)-1+N_{k'}+i'}}(z),x_{i'}^{k'})+d(x',x_{i'}^{k'})+d(x,x')\\
		&\leq 2\eps_{k'}+\eps_{k'}+\eps_{k}\\
		&\leq 4\eps_{k}.
	\end{split}
\end{equation*}
Then $orb(z,f)$ is $4\varepsilon_{k}$-dense in $X$ for any $k \geq 1,$ we obtain $z \in Trans$.
Note that $z\in G\subseteq G_1\subseteq \overline{B}(x_0,2\eps_1)\subseteq \overline{B}(x_0,\eps_0)\subseteq U.$ So we have $G\subseteq Trans\cap U$.

\textbf{Proof of Item $(2)$}: 
We define the stretched sequence $\{\gamma_{m}'\}_{m=1}^{\infty}$ by
\begin{equation}
	\gamma_{m}':=\gamma_{k} \quad \mathrm{if} \quad M_{N_{1}+N_{2}+\dots+N_{k-1}+t_{1}+\dots+t_{k-1}}+1\leq m\leq M_{N_{1}+N_{2}+\dots+N_{k}+t_{1}+\dots+t_{k}}.
\end{equation}
Then the sequence $\{\gamma_{m}'\}$ has the same limit-point set as the sequence of $\{\gamma_{k}\}$. If for any $z\in G,$ one has 
\begin{equation}\label{equation-DJ}
	\lim_{m\rightarrow \infty}d(\mathcal{E}_{m}(z),\gamma_{m}')=0,
\end{equation}
then the two sequences $\{\mathcal{E}_{m}(z)\}$, $\{\gamma_{m}'\}$ have the same limit-point set. Thus, we can obtain $G\subseteq G_{K}$ by (\ref{equation-D}). Meanwhile, it follows from (\ref{equation-A}) that $\lim\limits_{j\rightarrow \infty}\frac{M_{j+1}}{M_{j}}=1$. So from the definition of $\{\gamma_{m}'\}$, to verify (\ref{equation-DJ}) it is sufficient to prove that
\begin{equation}\label{equ:item1}
	\lim_{j\rightarrow \infty}d(\mathcal{E}_{M_{j}}(z),\gamma_{M_{j}}')=0.
\end{equation}
To this end, noting that for any $j\in\mathbb{N}$, there exists a unique $k\in\mathbb{N}$ such that
$$
N_{1}+N_{2}+\dots+N_{k}+t_{1}+\dots+t_{k}+1\leq j\leq N_{1}+N_{2}+\dots+N_{k+1}+t_{1}+\dots+t_{k+1}.
$$
Denote $\mathcal{L}_k=M_{N_{1}+N_{2}+\dots+N_{k}+t_{1}+\dots+t_{k}}.$ Since $z\in G\subseteq G_{k+1},$ there is $$(\mathbf{y}_1,\dots,\mathbf{y}_{k+1})=(y_1^1,\dots,y_{N_1}^1,\dots,y_1^{k+1},\dots,y_{N_{k+1}}^{k+1})\in \Gamma_1'\times\dots\times\Gamma_{k+1}'$$ such that $z\in G(\mathbf{y}_1,\dots,\mathbf{y}_{k},2\eps_{k}).$  We split into two cases to discuss.

Case (1): If $j=N_{1}+N_{2}+\dots+N_{k}+t_{1}+\dots+t_{k}+q$ with $1\leq q \leq N_{k+1}$, then by Lemma \ref{measure distance} and (\ref{equation-B}) one has
\begin{equation}\label{1}
	\begin{split}
		&d(\mathcal{E}_{M_{j}-\mathcal{L}_k}(f^{\mathcal{L}_k}(z)),\gamma_{k+1}) \\ 
		\leq~~&d(\mathcal{E}_{M_{j}-\mathcal{L}_k}(f^{\mathcal{L}_k}(z)),\frac{1}{q}\sum_{i=1}^{q}\mathcal{E}_{n_{k+1}}(y^{k+1}_{i}))+d(\frac{1}{q}\sum_{i=1}^{q}\mathcal{E}_{n_{k+1}}(y^{k+1}_{i}),\gamma_{k+1}) \\
		\leq~~& 2\eps_{k+1}+2\zeta_{k+1}+\frac{1}{q}\sum_{i=1}^{q}d(\mathcal{E}_{n_{k+1}}(y^{k+1}_{i}),\gamma_{k+1})\\
		\leq~~& 2\eps_{k+1}+3\zeta_{k+1}.
	\end{split}
\end{equation}

Case (2): If $j=N_{1}+N_{2}+\dots+N_{k+1}+t_{1}+\dots+t_{k}+q$ with $1\leq q \leq t_{k+1},$ then one has
\begin{equation}\label{AC}
	\begin{split}
		&d(\mathcal{E}_{M_{j}-\mathcal{L}_k}(f^{\mathcal{L}_k}(z)),\gamma_{k+1})  \\ \leq~~&\frac{M_{\sum_{i=1}^{k}(N_i+t_i)+N_{k+1}}-\mathcal{L}_k}{M_{j}-\mathcal{L}_k}d\left(\mathcal{E}_{M_{\sum_{i=1}^{k}(N_i+t_i)+N_{k+1}}-\mathcal{L}_k}\left(f^{\mathcal{L}_k}(z)\right),\gamma_{k+1}\right)\\
		&+\frac{M_{j}-M_{\sum_{i=1}^{k}(N_i+t_i)+N_{k+1}}}{M_{j}-\mathcal{L}_k}\times 1\\
		\leq~~& 1\times (2\eps_{k+1}+3\zeta_{k+1})+\frac{(t_{k+1}-1)K_{k+1}+K_{k+2}}{N_{k+1}(n_{k+1}+K_{k+1}-1)} \\
		\leq~~& 2\eps_{k+1}+3\zeta_{k+1}+\frac{(t_{k+1}-1)K_{k+1}+K_{k+2}}{n_{k+1}}\\
		\leq~~& 2\eps_{k+1}+4\zeta_{k+1}.
	\end{split}
\end{equation}
by (\ref{1}) and (\ref{equation-B}).
Meanwhile, by using (\ref{1}), one also has
\begin{equation}\label{AB}
	\begin{split}
		&d\left(\mathcal{E}_{M_{\sum_{i=1}^{k-1}(N_i+t_i)+N_{k}}-\mathcal{L}_{k-1}}\left(f^{\mathcal{L}_{k-1}}(z)\right),\gamma_{k+1}\right)\\ \leq~~&d\left(\mathcal{E}_{M_{\sum_{i=1}^{k-1}(N_i+t_i)+N_{k}}-\mathcal{L}_{k-1}}\left(f^{\mathcal{L}_{k-1}}(z)\right),\gamma_{k}\right)+d(\gamma_{k+1},\gamma_{k})\\
		\leq~~& 2\eps_{k}+3\zeta_{k}+d(\gamma_{k+1},\gamma_{k}).
	\end{split}
\end{equation}
Thus,
\begin{align*}
	&d(\mathcal{E}_{M_{j}}(z),\gamma_{M_{j}}')=d(\mathcal{E}_{M_{j}}(z),\gamma_{k+1})\\
	\leq~~&\frac{\mathcal{L}_{k-1}}{M_{j}}d(\mathcal{E}_{\mathcal{L}_{k-1}}(z),\gamma_{k+1})+\frac{M_{\sum_{i=1}^{k-1}(N_i+t_i)+N_{k}}-\mathcal{L}_{k-1}}{M_{j}}d\left(\mathcal{E}_{M_{\sum_{i=1}^{k-1}(N_i+t_i)+N_{k}}-\mathcal{L}_{k-1}}\left(f^{\mathcal{L}_{k-1}}(z)\right),\gamma_{k+1}\right)\\
	&+\frac{\mathcal{L}_{k}-M_{\sum_{i=1}^{k-1}(N_i+t_i)+N_{k}}}{M_{j}}d(\mathcal{E}_{\mathcal{L}_{k}-M_{\sum_{i=1}^{k-1}(N_i+t_i)+N_{k}}}(f^{M_{\sum_{i=1}^{k-1}(N_i+t_i)+N_{k}}}(z)),\gamma_{k+1})\\
	&+\frac{M_{j}-\mathcal{L}_k}{M_{j}}d(\mathcal{E}_{M_{j}-\mathcal{L}_k}(f^{\mathcal{L}_k}(z)),\gamma_{k+1})\\
	\leq~~& \frac{\mathcal{L}_{k-1}}{\mathcal{L}_k}\times 1 + 1\times (2\eps_{k}+3\zeta_{k}+d(\gamma_{k+1},\gamma_{k})) + \frac{(t_{k}-1)K_{k}+K_{k+1}}{n_{k}}\times 1+ 1\times(2\eps_{k+1}+4\zeta_{k+1})\\
	\leq~~& \zeta_{k}+2\eps_{k}+3\zeta_{k}+d(\gamma_{k+1},\gamma_{k})+\zeta_{k}+ 2\eps_{k+1}+4\zeta_{k+1}\\
	=~~	  &9\zeta_{k}+4\eps_{k}+d(\gamma_{k},\gamma_{k+1}).
\end{align*}
The second inequality is followed by using (\ref{1}), (\ref{AB}) and (\ref{AC}). The last inequality is followed by using (\ref{equation-A}) and (\ref{equation-B}). Taking $k$ to infinity on both sides, and note that LHS goes to zero by (\ref{equation-D}), this directly yields \eqref{equ:item1}, and the proof of item $(2)$ is completed.

\textbf{Proof of Item $(3)$}: Next we will prove
\begin{equation}\label{equ:packinghardpart}
	h_{top}(f,G)\geq \inf\{h_\mu (f):\mu\in K\}-4\eta.
\end{equation}
To this end, define $\mu_{k}:=\frac{1}{\sharp F_k}\sum_{x\in F_{k}}\delta_{x}$. Suppose $\mu=\lim\limits_{n\to\infty}\mu_{k_s}$ for some $k_s\to \infty$. For any fixed $p\in\mathbb{N^{+}}$ and all $p'\geq 0$, 
since $F_{p+p'}\subseteq G_{p+p'},$ one has $\mu_{p+p'}(G_{p+p'})=1.$ Combining with $G_{p+p'}\subseteq G_{p},$ we have $\mu_{p+p'}(G_{p})=1.$
Then $\mu(G_{p})\geq \limsup\limits_{n\to\infty}\mu_{k_s}(G_{p})=1$. It follows that
$\mu(G)=\lim\limits_{p\to \infty}\mu(G_p)=1$.

For $k \geqslant 1, i=1,2, \cdots, N_{k}$, let
\begin{equation}\label{equation_1}
	m_{N_{1}+\cdots+N_{k-1}+i}:=N_{1}+\cdots+N_{k-1}+t_{1}+\cdots+t_{k-1}+i.
\end{equation}
Then for any $p\geq 1,$ there is $k\geq 1$ and $1\leq i\leq N_k$ such that $m_p=N_{1}+\cdots+N_{k-1}+t_{1}+\cdots+t_{k-1}+i.$ Thus we have 
\begin{equation}\label{equation-AD}
	M_{m_{p+1}}\leq M_{m_p}+\max\{n_k,(t_k-1)K_k+2K_{k+1}+n_{k+1}\}.
\end{equation}

Fix $z\in G$ and $n\geq M_{m_{N_1+1}},$ there exists $k'\geq 2$ and $1\leq i'\leq N_{k'}$ such that $$M_{m_{N_{1}+\cdots+N_{k'-1}+i'}}\leq n\leq M_{m_{N_{1}+\cdots+N_{k'-1}+i'+1}}.$$
Since $z\in G\subseteq G_{k'},$ then there exists $\mathbf{y}_{i}^0=(y_1^{i,0},\dots,y_{N_{i}}^{i,0})\in \Gamma_{i}'$ for each $1\leq i\leq k'$ such that
$$z\in G_{k'}(\mathbf{y}_1^0,\dots,\mathbf{y}_k^0,2\eps_k).$$
Note that for any $k_s\geq k'$ and any $(\mathbf{y}_1,\dots,\mathbf{y}_{k_s})\in  \Gamma_1'\times\dots\times\Gamma_{k_s}',$ if $(y_1^{1},\dots,y_{N_{1}}^{1},\dots,y_1^{k'},\dots,y_{i'}^{k'})\neq (y_1^{1,0},\dots,y_{N_{1}}^{1,0},\dots,y_1^{k',0},\dots,y_{i'}^{k',0}),$ then $z$ and 
$z_{k_s}(\mathbf{y}_1,\dots,\mathbf{y}_{k_s})\in F_{k_s} $ are $(M_{\sum_{i=1}^{k'-1}(N_i+t_i)+i'},\frac{1}{3}\eps^*)$-separated by (\ref{equation-AB}). That is $z_{k_s}(\mathbf{y}_1,\dots,\mathbf{y}_{k_s})\not\in B_{M_{\sum_{i=1}^{k'-1}(N_i+t_i)+i'}}(z,\frac{1}{3}\eps^*)$. Then we have
\begin{equation}\label{equation-AG}
	\begin{split}
		\mu_{k_{s}}(B_{n}(z,\frac{1}{3}\eps^*))\leq&\mu_{k_{s}}(B_{M_{\sum_{i=1}^{k'-1}(N_i+t_i)+i'}}(z,\frac{1}{3}\eps^*))\\
		=& \frac{1}{|\Gamma_{1}'|\times|\Gamma_{2}'|\times\dots\times|\Gamma_{k_s}'|}\times(|\Gamma_{k'}|^{N_{k'}-i'}\times |\Gamma_{k'+1}'|\times\dots\times|\Gamma_{k_s}'|)\\
		=&\frac{1}{|\Gamma_{1}|^{N_1}\times|\Gamma_{2}|^{N_2}\times\dots\times|\Gamma_{k'-1}|^{N_{k'-1}}\times|\Gamma_{k'}|^{i'}}\\
		\leq& e^{-(\sum_{i=1}^{k'-1}n_iN_i+i'n_{k'})H^*}
	\end{split}
\end{equation}
by (\ref{equation-AC}). Note that
\begin{equation}
	\begin{split}
		\frac{n-(\sum_{i=1}^{k'-1}n_iN_i+i'n_{k'})}{\sum_{i=1}^{k'-1}n_iN_i+i'n_{k'}}&\leq \frac{M_{m_{N_{1}+\cdots+N_{k'-1}+i'+1}}-(\sum_{i=1}^{k'-1}n_iN_i+i'n_{k'})}{\sum_{i=1}^{k'-1}n_iN_i+i'n_{k'}}\\
		&\leq \frac{M_{m_{N_{1}+\cdots+N_{k'-1}+i'}}-(\sum_{i=1}^{k'-1}n_iN_i+i'n_{k'})}{\sum_{i=1}^{k'-1}n_iN_i+i'n_{k'}}\\
		&+\frac{\max\{n_{k'},(t_{k'}-1)K_{k'}+2K_{k'+1}+n_{k'+1}\}}{\sum_{i=1}^{k'-1}n_iN_i+i'n_{k'}}\\
		&\leq\frac{\sum_{i=1}^{k'-1}(N_i+t_i)K_i+i'K_{k'}}{\sum_{i=1}^{k'-1}n_iN_i+i'n_{k'}}+\zeta_{k'-1}\\
		&\leq \frac{\sum_{i=1}^{k'-2}(N_i+t_i)K_i}{n_{k'-1}N_{k'-1}}+ \frac{(N_{k'-1}+t_{k'-1})K_{k'-1}}{n_{k'-1}N_{k'-1}}+\frac{K_{k'}}{n_{k'}}+\zeta_{k'-1}\\
		&\leq \zeta_{k'-1}+ \frac{K_{k'-1}}{n_{k'-1}}+ \frac{t_{k'-1}K_{k'-1}}{n_{k'-1}N_{k'-1}}+\zeta_{k'}+\zeta_{k'-1}\\
		&\leq \zeta_{k'-1}+ \zeta_{k'-1}+ \zeta_{k'-1}+\zeta_{k'}+\zeta_{k'-1}\\
		&\leq 5\zeta_{1}
	\end{split}
\end{equation}
by (\ref{equation-AD}), (\ref{equation-A}) and (\ref{equation-B}), then if $H^*-\eta>0,$ by (\ref{equation-AF}) we have 
\begin{equation}\label{equation-AH}
	(H^*-\eta)(n-(\sum_{i=1}^{k'-1}n_iN_i+i'n_{k'}))<\eta(\sum_{i=1}^{k'-1}n_iN_i+i'n_{k'}),
\end{equation}
if $H^*-\eta\leq0,$ (\ref{equation-AH}) also holds. So we have
\begin{equation*}
	(\sum_{i=1}^{k'-1}n_iN_i+i'n_{k'})H^*>n(H^*-\eta).
\end{equation*}
Combining with (\ref{equation-AG}),  we have
\begin{equation*}
	\mu_{k_{s}}(B_{n}(z,\frac{1}{3}\eps^*))<e^{-n(H^*-\eta)}.
\end{equation*}
Then
\begin{equation}
	\mu(B_{n}(z,\frac{1}{3}\eps^*))\leq\liminf_{s\rightarrow \infty}\mu_{k_{s}}(B_{n}(z,\frac{1}{3}\eps^*))
	\leq e^{-n(H^*-\eta)}.
\end{equation}
Using (\ref{equation-GI}), we have
\begin{equation}
	\underline{h}_{\mu}(f,z)
	\geq \liminf_{n\to\infty}-\frac{1}{n}\log\mu(B_{n}(z,\frac{1}{3}\eps^*)) \\
	\geq H^{*}-\eta.
\end{equation}
This implies $\underline{h}_{\mu}(f)=\int\underline{h}_{\mu}(f,x)d\mu\geq H^{*}-\eta.$
Taking supermum over all $\nu \in \mathcal{M}(X)$ with $\nu(G)=1,$ and applying Lemma \ref{lemma-aa}, we eventually have
$$h_{top}(f,G)\geq H^{*}-\eta= \inf\{h_{\gamma_k}(f):j\in\mathbb{N^{+}}\}-2\eta\geq \inf\{h_\mu(f):\mu\in K\}-4\eta.$$
By the arbitrariness of $\eta$ and $G\subseteq G_{K}\cap U\cap Trans,$ this yields
$$h_{top}(f,G_{K}\cap U\cap Trans)\geq \inf\{h_\mu (f):\mu\in K\}$$
which completes the proof of Theorem \ref{maintheorem-2}(1).
\qed

\subsection{ Proof of Theorem \ref{maintheorem-2}(2)} 
Since $K\subseteq  \cM_{f|_{X_{n_0}}}(X_{n_0})$ is a non-empty compact connected set, by Lemma \ref{lemma-A},  there exists a sequence $\left\{\beta_{1}, \beta_{2}, \cdots\right\}$ in $K$ such that
$$
\overline{\left\{ \beta_{j}: j \in \mathbb{N}^+, j>n\right\}}=K, \forall n \in \mathbb{N}^+ \text { and } \lim _{j \rightarrow \infty} d\left( \beta_{j},  \beta_{j+1}\right)=0.
$$
Note that $(X_{n_0},f|_{X_{n_0}})$ has uniform separation property and $\beta_j\in \cM_{f|_{X_{n_0}}}(X_{n_0})$ for any $j\in\mathbb{N^{+}}.$
Then using the same method as Theorem \ref{maintheorem-2}(1) for $\left\{\beta_{1}, \beta_{2}, \cdots\right\}$, we have $$G_{K}\cap U\cap Trans\neq\emptyset,\text{ and }h_{top}(f,G_{K}\cap U\cap Trans\})=\inf\{h_{\mu}(f):\mu\in K\},$$ which completes the proof of Theorem \ref{maintheorem-2}(2). \qed

\subsection{Proof of Theorem \ref{maintheorem-3}(1)}

\subsubsection{Periodic decomposition}

Suppose we have a collection $\mathcal{D}=\{D_{0},D_{1},\dots,D_{n-1}\}$ of subsets of
$X$ such that $f(D_{i}) \subseteq D_{i+1(\mathrm{mod}\ n)}$ for any $i \in \{1,\dots,n-1\}$. We call $\mathcal{D}$ a periodic orbit of sets. Clearly, the union of the $D_{i}$ is an invariant subset in these circumstances and $f^{n}$
is invariant on each $D_{i}$. Also, $f^{k}(D_{i}) \subseteq D_{i+k(\mathrm{mod}\ n)}$ for all $k \in \mathbb{N}$. We call $\mathcal{D}$ a periodic decomposition if the $D_{i}$ are all closed, $D_{i}\cap D_{j}$ is nowhere dense whenever $i\neq j$, and the union of the $D_{i}$ is $X$. The condition that the $D_{i}$ have nowhere dense overlap is equivalent to saying that their interiors are disjoint. In the special case where the $D_{i}$ are mutually disjoint, each $D_{i}$ is clopen since its complement is a finite union of closed sets. The number of sets in a decomposition will be called the length of the decomposition.
A closed set is called regular closed if it is the closure of its interior or, equivalently, it is the closure of an open set. A periodic decomposition is regular if all of its elements are regular closed. Since clopen sets are regular closed, a periodic decomposition is always regular in the case where $D_{i}\cap D_{j}$ is empty for $i\neq j$.

Recall that $(X,f)$ is mixing if for any non-empty open sets $U,V\subseteq X,$ there is $N\in\mathbb{N}$ such that $f^{-n}U\cap V\neq\emptyset$ for any $n\geq N.$
\begin{Lem}\label{AU}\cite[Lemma 3.16]{HT}
	Suppose that a dynamical system $(X,f)$ is transitive and has shadowing property. If there is a periodic point $p_{0}$ with period $l$, then there is a regular periodic decomposition $\mathcal{D}=\{D_{0},D_{1},\dots,D_{n-1}\}$ for some $n$ dividing $l$ such that $D_{i}\cap D_{j}$ is empty for $0\leq i\neq j\leq n-1$, $f^{n}$ is mixing on each $D_{i}$ and $f^{n}|_{D_{i}}$ has shadowing property for any $i \in \{0,1,\dots,n-1\}$.
\end{Lem}

\begin{Lem}\label{AY}\cite[Lemma 3.17]{HT}
	Suppose that $(X,f)$ is a dynamical system. Let $Y_1\subseteq Y_2\subseteq X$ be two non-empty compact $f$-invariant subsets. If for each $i\in\{1,2\},$ $(Y_i,f|_{Y_i})$ is transitive and has shadowing property, there is a regular periodic decomposition $\mathcal{D}_i=\{D_{0}^i,D_{1}^i,\dots,D_{n-1}^i\}$ for $(Y_i,f|_{Y_i})$ with $D_{j_1}\cap D_{j_2}=\emptyset$ for $0\leq j_1\neq j_2\leq n-1$, and $D_0^1\cap D_0^2\neq \emptyset,$ then we have $D_j^1\subseteq D_j^2$ for any $0\leq j\leq n-1.$
\end{Lem}

Let $k\geq 1.$ For $\eps>0$ and $n\in\mathbb{N^{+}}$,  two points $x$ and $y$ are $(n,  \eps)$-separated for $(X,f^k)$ if there is $0\leq i\leq n-1$ such that $d(f^{ik}(x),f^{ik}(y))>\eps.$
A subset $E$ is $(n,  \eps)$-separated for $(X,f^k)$ if any pair of different points of $E$ are $(n,  \eps)$-separated for $(X,f^k)$. 
\begin{Lem}\label{lemma-separation-2}
	Suppose that $(X,f)$ is a dynamical system. Let $k\geq 1.$ Then for any $\varepsilon>0,$ there exists $\varepsilon'>0$ such that if $x,y\in X$ are $(nk,\varepsilon)$-separated, then $x,y$ are $(n,\varepsilon')$-separated for $(X,f^k).$
\end{Lem}
\begin{proof}
	Since $f$ is continuous, for any $\varepsilon>0,$ there exists $\varepsilon'>0$ such that if $d(x,y)<\varepsilon',$ then one has $d(f^i(x),f^i(y))<\varepsilon$ for any $0\leq i\leq n-1.$ Then if $x,y\in X$ are $(nk,\varepsilon)$-separated, then $x,y$ are $(n,\varepsilon')$-separated for $(X,f^k).$ Otherwise, we have $d_{nk}(x,y)<\varepsilon.$
\end{proof}

\subsubsection{Proof of Theorem \ref{maintheorem-3}(1)}
Since $G_{K}\cap U\cap Trans\subseteq G_K$ and $h_{top}(f,G_K)\leq\inf\{h_\mu (f):\mu\in K\}$ by Lemma \ref{lem-Bowen},   one has 
\begin{equation*}
	h_{top}(f,G_{K}\cap U\cap Trans)\leq \inf\{h_\mu (f):\mu\in K\}.
\end{equation*}
So it remains to show that
\begin{equation*}
	h_{top}(f,G_{K}\cap U\cap Trans)\geq \inf\{h_\mu (f):\mu\in K\}.
\end{equation*}

Since $\mathrm{Per}(f|_{X_{1}})$ $\neq \emptyset,$ there is a periodic point $p_{0}$ in $X_1$ with period $Q$ for some $Q\in\mathbb{N^{+}}.$
By Lemma \ref{AU}, for any $l\in \mathbb{N^{+}}$, there is a periodic decomposition $\mathcal{D}_{l}=\{D^{l}_{0},D^{l}_{1},\dots,D^{l}_{n_{l}-1}\}$ for some $n_{l}$ dividing $Q$ such that $D^{l}_{i}\cap D^{l}_{j}$ is empty for $0\leq i\neq j\leq n_l$, $f^{n_{l}}$ is mixing on each $D^{l}_{i}$ and $f^{n_{l}}|_{D^{l}_{i}}$ has shadowing property for any $i \in \{0,1,\dots,n_{l}-1\}$. Using the pigeon-hole principle we can assume that there exists $k$ dividing $Q$ such that $n_{l}=k$ for any $l \in \mathbb{N^{+}}.$
We can also assume that $p_{0} \in D_{0}^{l}$ for any $l \in \mathbb{N^{+}}$. Then $f^{i}(p_{0}) \in D_{i}^{l}$ for any $l \in \mathbb{N^{+}}$ and $i \in \{0,1,\dots,k-1\}$.
By Lemma \ref{AY}, we have
$D_{i}^{l} \subseteq D_{i}^{l+1}$ for any $l \in \mathbb{N^{+}}$, $i \in \{0,1,\dots,k-1\}$.
Let $Y_{i}=\overline{\bigcup_{l\geq 1}D_{i}^{l}}$ for any $i \in \{0,1,\dots,k-1\}$. Then $X=\bigcup_{0\leq i \leq k-1}Y_{i}$ and $f^{j}(Y_{i})=Y_{i+j\ (\mathrm{mod} \ n)}$. For any $l \in \mathbb{N^{+}},$ $f^{k}$ is mixing on each $D^{l}_{0}$ and $f^{k}|_{D^{l}_{0}}$ has shadowing property, so $(D_{0}^{l},f^{k}|_{D_{0}^{l}})$ has specification property by Proposition \ref{prop-specif}.

For any $\mu \in \mathcal{M}(X_{l})$ and $l \in \mathbb{N^{+}}$, define $h_{*}^{l}(\mu) \in \mathcal{M}(D_{0}^{l})$ by: $h_{*}^{l}(\mu)(A)=\mu (A \cup f(A) \cup \dots \cup f^{k-1}(A))$, where $A$ ia a Borel set of $D_{0}^{l}$. By \cite[Proposition 23.17]{Sig}, $h_{*}^{l}$ is a homeomorphism from $\mathcal{M}_{f|_{X_{l}}}(X_{l})$ onto $\mathcal{M}_{f^{k}|_{D_{0}^{l}}}(D_{0}^{l})$ and $(h_{*}^{l})^{-1}(\nu)=\frac{1}{k}(\nu + f_{*}\nu +\dots + f^{k-1}_{*}\nu) \in \mathcal{M}_{f|_{X_{l}}}(X_{l})$ for any $\nu \in \mathcal{M}_{f^{k}|_{D_{0}^{l}}}(D_{0}^{l})$ where $f_*\nu(B)=\nu(f^{-1}(B))$ for any Borel set $B$.
Since $h^{l+1}_{*}|_{\mathcal{M}_{f|_{X_{l}}}(X_{l})}=h^{l}_{*}$ for any $l \in \mathbb{N^{+}}$, we can define $h_{*}$ in $\bigcup_{l\geq 1}\mathcal{M}_{f|_{X_{l}}}(X_l)$ such that $h_{*}|_{\mathcal{M}_{f|_{X_{l}}}(X_{l})}=h^{l}_{*}$. Then $h_{*}$ is a homeomorphism from $\bigcup_{l\geq 1}\mathcal{M}_{f|_{X_{l}}}(X_l)$ onto $\bigcup_{l\geq 1}\mathcal{M}_{f^{k}|_{D_{0}^{l}}}(D_{0}^{l})$, and for any $\nu\in \bigcup_{l\geq 1}\mathcal{M}_{f^{k}|_{D_{0}^{l}}}(D_{0}^{l}),$ we have 
\begin{equation}\label{BC}
	\begin{split}
		kh_{(h_{*}^{l})^{-1}(\nu)}(f)&=h_{(h_{*}^{l})^{-1}(\nu)}(f^k)\\
		&=h_{\frac{1}{k}(\nu + f_{*}\nu +\dots + f^{k-1}_{*}\nu)}(f^k)\\
		&=\frac{1}{k}(h_{\nu}(f^k)+h_{f_{*}\nu}(f^k)+\dots +h_{f^{k-1}_{*}\nu}(f^k))\\
		&=\frac{1}{k}(kh_{\nu}(f^k))\\
		&=h_{\nu}(f^k)
	\end{split}
\end{equation}

Let $\eta>0.$ Since $(X,f)$ has uniform separation property, by Lemma \ref{lem-n-eps} there exists $\eps^*>0$ so that for $\mu$ ergodic and any neighbourhood $F\subseteq \cM(X)$ of $\mu$,   there exists $n_{F,  \mu,  \eta}^*\in\N$ such that for any $n\geq n_{F,  \mu,  \eta}^*$,   there is a $(n,  \eps^*)$-separated set $\Gamma_n\subseteq X_{n,  F}\cap S_\mu$ with
\begin{equation}\label{BA}
	|\Gamma_n|\geq e^{n(h_{\mu}(f)-\frac{\eta}{2k})}.
\end{equation}
Take $\tilde{n}_{F,  \mu,  \eta}^*>n_{F,  \mu,  \eta}^*$ large enough such that for any $n\geq \tilde{n}_{F,  \mu,  \eta}^*$ we have
\begin{equation}\label{BE}
	\frac{1}{k}e^{n(h_{\mu}(f)-\frac{\eta}{2k})}\geq e^{(n+k)(h_{\mu}(f)-\frac{\eta}{k})}.
\end{equation}
If $\mu\in \mathcal{M}_{f|_{X_{l}}}(X_l)$ for some $l\geq 1,$ then $\Gamma_n\subseteq X_{n,  F}\cap S_\mu\subseteq X_{n,  F}\cap X_l.$  Using the pigeon-hole principle, for any $n\geq \tilde{n}_{F,  \mu,  \eta}^*$ there exists $\tilde{\Gamma}_n\subseteq \Gamma_n$ and $0\leq i\leq k-1$ such that $\tilde{\Gamma}_n\subseteq D_i^l$ and  
\begin{equation}\label{BH}
	|\tilde{\Gamma}_n|\geq \frac{1}{k}e^{n(h_{\mu}(f)-\frac{\eta}{2k})}.
\end{equation}
Let $\tilde{\tilde{\Gamma}}_n=\{f^{-i}(x):x\in \tilde{\Gamma}_n\}.$ Then $\tilde{\tilde{\Gamma}}_n\subseteq D_0^l$ is a $(n+k,  \eps^*)$-separated set and  
\begin{equation*}
	|\tilde{\tilde{\Gamma}}_n|=	|\tilde{\Gamma}_n|\geq e^{(n+k)(h_{\mu}(f)-\frac{\eta}{k})}.
\end{equation*}
by (\ref{BE}) and (\ref{BH}). Let $\tilde{\tilde{n}}_{F,  \mu,  \eta}^*=\tilde{n}_{F,  \mu,  \eta}^*+k.$ Then for any $n\geq \tilde{\tilde{n}}_{F,  \mu,  \eta}^*$,   there is a $(n,  \eps^*)$-separated set $\tilde{\tilde{\Gamma}}_n\subseteq X_{n,  F}\cap S_\mu\cap D_0^l$ with
\begin{equation}\label{BI}
	|\tilde{\tilde{\Gamma}}_n|\geq e^{n(h_{\mu}(f)-\frac{\eta}{k})}.
\end{equation}
Since $h_{*}$ is continuous and $h_*(\frac{1}{nk}\sum_{j=0}^{nk-1}\delta_{f^{j}(x)})=\frac{1}{n}\sum_{j=0}^{n-1}\delta_{f^{jk}(x)}$ for any $x\in D_0^l,$ then
for any $\nu\in \mathcal{M}_{f^{k}|_{D_{0}^{l}}}(D_{0}^{l}),$ and any neighbourhood $F'\subseteq \cM(Y_0)$ of $\nu,$ there is a neighbourhood $F\subseteq \cM(X)$ of $(h_{*})^{-1}(\nu)$ such that if $x\in D_0^i$ and $\frac{1}{nk}\sum_{j=0}^{nk-1}\delta_{f^{j}(x)}\in F,$ then $\frac{1}{n}\sum_{j=0}^{n-1}\delta_{f^{jk}(x)}\in F'.$ Then by (\ref{BC}), (\ref{BI}) and Lemma \ref{lemma-separation-2} there exists $\eps'>0$ so that for $\nu\in \bigcup_{l\geq 1}\mathcal{M}_{f^{k}|_{D_{0}^{l}}}(D_{0}^{l})$ ergodic and any neighbourhood $F'\subseteq \cM(Y_0)$ of $\nu$,   there exists $n_{F',  \nu,  \eta}'\in\N$ such that for any $n\geq n_{F',  \nu,  \eta}'$,   there is a $(n,  \eps')$-separated set $\Gamma_n'\subseteq X_{n,  F'}\cap S_\nu\cap (\bigcup_{l\geq 1}D_0^l)$ for $(X,f^k)$ with
\begin{equation}\label{BB}
	|\Gamma_n'|\geq e^{nk(h_{(h_{*}^{l})^{-1}(\nu)}(f)-\frac{\eta}{k})} = e^{n(h_{\nu}(f^k)-\eta)}.
\end{equation}

By Lemma \ref{lemma-A},  there exists a sequence $\left\{\alpha_{1}, \alpha_{2}, \cdots\right\}$ in $K$ such that
$$
\overline{\left\{\alpha_{j}: j \in \mathbb{N}^+, j>n\right\}}=K, \forall n \in \mathbb{N}^+ \text { and } \lim _{j \rightarrow \infty} d\left(\alpha_{j}, \alpha_{j+1}\right)=0.
$$
For any $j \in \mathbb{N}^+$, by Lemma \ref{lemma-MM}, there exists $\beta_j\in \cM_{f|_{X_{l_k}}}(X_{l_j})$ for some $l_j \in \mathbb{N^{+}}$ such that $d(\beta_j,\alpha_j)<\varepsilon_{k}$ and $h_{\beta_j}(f)\geq h_{\alpha_j}(f)-\frac{\eta}{k}.$
By applying Proposition \ref{proposition of entropy-dense property} to $(X_{l_j},f)$, there exists $\gamma_j\in \cM^e_{f|_{X_{l_j}}}(X_{l_j})$ such that $d(\gamma_j,\beta_j)<\varepsilon_{j}$ and $h_{\gamma_j}(f)\geq h_{\beta_j}(f)-\frac{\eta}{k}.$ Then we have
\begin{equation}\label{BJ}
	\overline{\left\{\gamma_{j}: j \in \mathbb{N}^+, j>n\right\}}=\overline{\left\{\alpha_{j}: j \in \mathbb{N}^+, j>n\right\}}=K, \forall n \in \mathbb{N}^+\text { and } \lim _{j \rightarrow \infty} d\left(\gamma_{j}, \gamma_{j+1}\right)=0.
\end{equation}
For any $j \in \mathbb{N}^+,$ let $\mu_j=h_{*}(\gamma_j)\in\mathcal{M}_{f^{k}|_{D_{0}^{l_j}}}(D_{0}^{l_j}).$

For any non-empty open set $U\subseteq X$, there exists $i_{0} \in \{0,1,\dots,k-1\}$ such that $U \cap Y_{i_{0}} \neq \emptyset$. Then $f^{-i_{0}}(U)\cap Y_{0}\neq \emptyset$. Let $V=f^{-i_{0}}(U)\cap Y_{0}$, then $V$ is any non-empty open set for $Y_{0}$.

Using similar construction of Theorem \ref{maintheorem-2}(1) considering $(Y_{0},f^{k}|_{Y_0})$,  there exists $G\subseteq V$ such that
\begin{description}
	\item[(a)] $h_{top}(f^k,G)\geq \inf\{h_{\mu_j}(f^k):j\in\mathbb{N^{+}}\}-2\eta;$
	
	\item[(b)] for any $z\in G,$ $\{f^{ik}(z):i\in \mathbb{N}\}$ is dense in $Y_{0};$
	
	\item[(c)] for any $z\in G,$ two sequences $\{\mathcal{E}_{mk}(z)\}_{m=1}^{\infty},$ $\{\mu_{m}\}_{m=1}^{\infty}$ have the same limit-point set. More precisely.
	\begin{description}
		\item[(c1)] For any sequence $\{\chi_{m}\}_{m=1}^{\infty} \subseteq \mathbb{N}$, there is a sequence $\{n_{m}\}_{m=1}^{\infty}\subseteq \mathbb{N}$ such that $$\lim_{m\to \infty }d(\frac{1}{n_{m}}\sum_{j=0}^{n_{m}-1}\delta_{f^{jk}(z)},\mu_{\chi _{m}})=0;$$
		\item[(c2)] For any $\{n_{m}\}_{m=1}^{\infty} \subseteq \mathbb{N}$, there exists a sequence $\{\chi_{m}\}_{m=1}^{\infty} \subseteq \mathbb{N}$ such that $$\lim_{m\to \infty }d(\frac{1}{n_{m}}\sum_{j=0}^{n_{m}-1}\delta_{f^{jk}(x)},\mu_{\chi _{m}})=0.$$
	\end{description}
\end{description}
Then,
\begin{description}
	\item[(A)] by Proposition \ref{prop-AA}(2) and (\ref{BC}), we have
	\begin{equation}
		\begin{split}
			h_{top}(f,G)&=\frac{1}{k}h_{top}(f^k,G)\\
			&\geq \inf\{\frac{1}{k}h_{\mu_j}(f^k):j\in\mathbb{N^{+}}\}-\frac{2\eta}{k}\\
			&= \inf\{h_{\gamma_j}(f):j\in\mathbb{N^{+}}\}-\frac{2\eta}{k}\\
			&\geq \inf\{h_{\alpha_j}(f):j\in\mathbb{N^{+}}\}-\frac{4\eta}{k}\\
			&\geq\inf\{h_{\mu}(f):\mu\in K\}-\frac{4\eta}{k}.
		\end{split}
	\end{equation}
	
	\item[(B)] for any $z\in G,$ $\{f^{i}(z):i\in \mathbb{N}\}$ is dense in $X$.
	
	\item[(C)] for any $z\in G,$ two sequences $\{\mathcal{E}_{m}(z)\}_{m=1}^{\infty},$ $\{\gamma_{m}\}_{m=1}^{\infty}$ have the same limit-point set. This is obtained by the following:
	\begin{description}
		\item[(C1)] for any limit-point $\gamma$ of $\{\gamma_{m}\}_{m=1}^{\infty}$, there is a sequence $\{\chi_{m}\}_{m=1}^{\infty} \subseteq \mathbb{N}$ such that $$\lim\limits_{m\to\infty }d(\gamma_{\chi _{m}},\gamma)=0.$$ For the $\{\chi_{m}\}_{m=1}^{\infty}$, by item (c1) there is a sequence $\{n_{m}\}_{m=1}^{\infty}\subseteq \mathbb{N}$ such that $$\lim_{m\to \infty }d(\frac{1}{n_{m}}\sum_{j=0}^{n_{m}-1}\delta_{f^{jk}(z)},\mu_{\chi _{m}})=0.$$  Then $$\lim_{m\to\infty }d(\frac{1}{n_{m}}\sum_{j=0}^{n_{m}-1}\delta_{f^{jk+s}(z)},f_{*}^{s}\mu_{\chi _{m}})=\lim_{m\rightarrow \infty }d(f_{*}^{s}(\frac{1}{n_{m}}\sum_{j=0}^{n_{m}-1}\delta_{f^{jk}(z)}),f_{*}^{s}\mu_{\chi _{m}})=0$$ for any $s \in \{0,1,\dots,k-1\}$. Thus one has $$\lim_{m\rightarrow \infty }d(\frac{1}{n_{m}k}\sum_{j=0}^{n_{m}k-1}\delta_{f^{j}(z)},\gamma_{\chi _{m}})=\lim_{m\rightarrow \infty }d(\frac{1}{n_{m}k}\sum_{j=0}^{n_{m}k-1}\delta_{f^{j}(z)},\frac{1}{k}\sum_{i=0}^{k-1}f^{i}_{*}\mu_{\chi _{m}})=0.$$ So $\lim\limits_{m\to\infty }d(\frac{1}{n_{m}k}\sum\limits_{j=0}^{n_{m}k-1}\delta_{f^{j}(z)},\gamma)=0.$
		\item[(C2)]  for any limit-point $\nu$ of $\{\mathcal{E}_{m}(z)\}_{m=1}^{\infty},$ there is a sequence $\{N_{m}\} \subseteq \mathbb{N}$ such that $$\lim\limits_{m\rightarrow \infty }d(\frac{1}{N_{m}}\sum_{j=0}^{N_{m}-1}\delta_{f^{j}(z)},\nu)=0.$$ For the $\{N_{m}\}_{m=1}^{\infty},$ there exist $\{n_{m}\}_{m=1}^{\infty}$ and $\{r_{m}\}_{m=1}^{\infty}$ where $r_{m}\in \{0,1,\dots,k-1\}$ such that $N_{m}=n_{m}k+r_{m}$. Then $\lim\limits_{m\rightarrow \infty }d(\frac{1}{n_{m}k}\sum\limits_{j=0}^{n_{m}k-1}\delta_{f^{j}(z)},\frac{1}{N_{m}}\sum\limits_{j=0}^{N_{m}-1}\delta_{f^{j}(z)})=0$ by Lemma \ref{measure distance}. By the item (c2) there exists a sequence $\{\chi_{m}\} \subseteq \mathbb{N}$ such that $\lim\limits_{m\rightarrow \infty }d(\frac{1}{n_{m}}\sum\limits_{j=0}^{n_{m}-1}\delta_{f^{jk}(z)},\mu_{\chi _{m}})=0$. Then $$\lim\limits_{m\rightarrow \infty }d(\frac{1}{n_{m}k}\sum_{j=0}^{n_{m}k-1}\delta_{f^{j}(z)},\gamma_{\chi _{m}})=\lim_{m\rightarrow \infty }d(\frac{1}{n_{m}k}\sum_{j=0}^{n_{m}k-1}\delta_{f^{j}(z)},\frac{1}{k}\sum_{i=0}^{k-1}f^{i}_{*}\mu_{\chi _{m}})=0.$$ 
		So $\lim\limits_{m\rightarrow \infty }d(\gamma_{\chi _{m}},\nu)=0.$
	\end{description}
\end{description}
By item (C) and (\ref{BJ}), we have $G\subseteq G_K.$ Let $\tilde{G}=\{f^{i_{0}}(z):z\in G\}$, then $\tilde{G}\subseteq G_{K}\cap Trans \cap U$ and $h_{top}(f,\tilde{G})=h_{top}(f,G)\geq\inf\{h_{\mu}(f):\mu\in K\}-\frac{4\eta}{k}$ by Prposition \ref{prop-AA} (1). 
By the arbitrariness of $\eta,$ this yields
$$h_{top}(f,G_{K}\cap U\cap Trans)\geq \inf\{h_\mu (f):\mu\in K\}$$
which completes the proof of Theorem \ref{maintheorem-3}(1).\qed

\subsection{Proof of Theorem \ref{maintheorem-3}(2)}
Since $K\subseteq  \cM_{f|_{X_{n_0}}}(X_{n_0})$ is a non-empty compact connected set, by Lemma \ref{lemma-A},  there exists a sequence $\left\{\beta_{1}, \beta_{2}, \cdots\right\}$ in $K$ such that
$$
\overline{\left\{ \beta_{j}: j \in \mathbb{N}^+, j>n\right\}}=K, \forall n \in \mathbb{N}^+ \text { and } \lim _{j \rightarrow \infty} d\left( \beta_{j},  \beta_{j+1}\right)=0.
$$
Note that $(X_{n_0},f|_{X_{n_0}})$ has uniform separation property and $\beta_j\in \cM_{f|_{X_{n_0}}}(X_{n_0})$ for any $j\in\mathbb{N^{+}}.$
Then using the same method as Theorem \ref{maintheorem-3}(1) for $\left\{\beta_{1}, \beta_{2}, \cdots\right\}$, we have $$G_{K}\cap U\cap Trans\neq\emptyset,\text{ and }h_{top}(f,G_{K}\cap U\cap Trans\})=\inf\{h_{\mu}(f):\mu\in K\},$$ which completes the proof of Theorem \ref{maintheorem-3}(2). \qed

\section{Homoclinic class and hyperbolic ergodic measure: proofs of Theorem \ref{maintheorem-1'} and \ref{maintheorem-1}}\label{section-4}

\subsection{Some abtract results}
First, we introduct some abtract results for possibly more applications.

\begin{Def}
	Given a dynamical system $(X,f)$. We say $S\subseteq X$ is periodic if there is $x\in Per(f)$ such that  $S=\{f^{i}(x): i \in \mathbb{N} \}$. 
\end{Def}
\begin{Lem}\label{BD}\cite[Lemma 4.3]{HT}
	Suppose that a dynamical system $(X,f)$ with $Per(f)\neq \emptyset$ is transitive and has the
	shadowing property. If $X$ is not periodic, then $h_{top}(f,X)>0$.
\end{Lem}

\begin{Prop}\label{BG}\cite[Proposition 4.6]{HT}
	Suppose that $(X,f)$ is a dynamical system with a sequence of nondecreasing $f$-invariant compact subsets $\{X_{n} \subseteq X:n \in \mathbb{N^{+}} \}$ such that $\overline{\bigcup_{n\geq 1}X_{n}}=X$, $({X_{n}},f|_{X_{n}})$ has shadowing property and is transitive for any $n \in \mathbb{N^{+}}$, and $\mathrm{Per}(f|_{X_{1}})\neq \emptyset$. 
	Then $(X,f)$ has measure $\nu$ with full support $($i.e. $S_\nu=X$$)$ and $\nu(\bigcup_{n\geq 1}X_{n})=1$. Moreover, the set of such measures is dense in $\overline{\{\mu \in \mathcal{M}_{f}(X):\mu(\bigcup_{n\geq 1}X_{n})=1\}}$.
\end{Prop}
Let $C^{0}(X,\mathbb{R})$ denote the space of real continuous functions on $X$ with the norm $||f||:=\sup_{x\in X}|f(x)|.$
For any $\varphi\in C^{0}(X,\mathbb{R})$, define the $\varphi$-irregular set as
\begin{equation*}
	I_{\varphi}(f) := \left\{x\in X: \lim_{n\to\infty}\frac1n\sum_{i=0}^{n-1}\varphi(f^i(x)) \,\, \text{ diverges }\right\}.
\end{equation*}
It's easy to see that $IR(f)=\cup_{\varphi\in C^{0}(X,\mathbb{R})}I_{\varphi}(f).$

\begin{Prop}\label{proposition-AD}\cite[Proposition 4.9]{HT}
	Suppose that $(X,f)$ is a dynamical system. Let $Y\subseteq X$ be a non-empty compact $f$-invariant set. If there is $\phi_0\in C^{0}(Y,\mathbb{R})$ such that $I_{\phi_0}(f|_{Y})\neq\emptyset,$ the set $\mathcal C^*:=\{\phi\in C^{0}(X,\mathbb{R}):I_\phi(f)\cap Y\neq\emptyset\}$ is an open and dense subset in $C^{0}(X,\mathbb{R})$.
\end{Prop}

\begin{maintheorem}\label{maintheorem-6}
	Suppose that $f$ is a homeomorphism from $X$ onto $X$, $(X,f)$ has a sequence of nondecreasing $f$-invariant compact subsets $\{X_{n} \subseteq X:n \in \mathbb{N^{+}} \}$ such that $\overline{\bigcup_{n\geq 1}X_{n}}=X$, $\mathrm{Per}(f|_{X_{1}})$ $\neq \emptyset,$ $({X_{n}},f|_{X_{n}})$ has shadowing property and is transitive for any $n \in \mathbb{N^{+}}.$ If $X_{n_0}$ is not periodic and $(X_{n_0},f|_{X_{n_0}})$ has uniform separation property for some ${n_0}\in\mathbb{N^{+}}$, then for any $\varphi \in C^{0}(X,\mathbb{R})$ with $I_{\varphi}(f) \cap X_{n_0}\neq\emptyset$ and any non-empty open set $U\subseteq X,$ we have 
	$$h_{top}(f,I_{\varphi}(f)\cap \{x\in X:V_T(x)\subseteq \cM_{f|_{X_{n_0}}}(X_{n_0})\}\cap U\cap Trans\})\geq h_{top}(f,X_{n_0}).$$
	Moreover, if $(X,f)$ has uniform separation property, then we have
	$$h_{top}(f,I_{\varphi}(f)\cap \{x\in X: \exists\ \mu_1,\mu_{2}\in V_f(x)\text{ s.t. }S_{\mu_1}\subseteq X_{n_0}, S_{\mu_{2}}=X \}\cap U\cap Trans\})\geq h_{top}(f,X_{n_0}),$$
	$$h_{top}(f,I_{\varphi}(f)\cap \{x\in X:S_\mu=X \text{ for any }\mu \in V_T(x)\}\cap U\cap Trans\})\geq h_{top}(f,X_{n_0}).$$
	Moreover, the functions with  $I_{\varphi}(f) \cap X_{n_0}\neq\emptyset$ are open and dense in the space of continuous functions. 
\end{maintheorem}
\begin{proof}
	For any $\eta>0,$ by the variational principle, there is $\mu\in \cM_{f|_{X_{n_0}}}(X_{n_0})$ such that $h_{\mu}(f)\geq h_{top}(f,X_{n_0})-\eta.$ Since $I_{\varphi}(f) \cap X_{n_0}\neq\emptyset,$ there exist $\mu_1,\mu_2\in \cM_{f|_{X_{n_0}}}(X_{n_0})$ such that $\int \varphi d\mu_1< \int \varphi d\mu_2.$ Take $0< \theta<1$ close to 1 such that for any $i\in\{1,2\}$
	\begin{equation*}
		h_{\nu_i} (f)=\theta h_{\mu}(f)+(1-\theta)h_{\mu_i}(f) \geq h_{\mu}(f)-\eta,
	\end{equation*}
    where $\nu_i=\theta\mu+(1-\theta)\mu_i.$
    Then we have $$\int \varphi d\nu_1=\theta\int \varphi d\mu+(1-\theta)\int \varphi d\mu_1< \theta\int \varphi d\mu+(1-\theta)\int \varphi d\mu_2=\int \varphi d\nu_2,$$ and $\nu_1,\nu_2\in \cM_{f|_{X_{n_0}}}(X_{n_0}).$ 
    By Proposition \ref{BG}, there exists $\omega\in \mathcal{M}_{f}(X)$ with $S_\omega=X$ and $\omega(\bigcup_{n\geq 1}X_{n})=1$ such that 
    $$|\int \varphi d\omega-\int \varphi d\nu_2|<\frac{1}{2}\int \varphi d\nu_2-\frac{1}{2}\int \varphi d\nu_1,$$
    then
    $$\int \varphi d\nu_1<\frac{1}{2}\int \varphi d\nu_1+\frac{1}{2}\int \varphi d\nu_2< \int \varphi d\omega.$$ Take $0< \theta_1<\theta_2<1$ close to 1 such that for any $i\in\{1,2\}$
    \begin{equation*}
    	h_{\omega_i} (f)=\theta_i h_{\nu_1}(f)+(1-\theta_i)h_{\omega}(f) \geq h_{\mu}(f)-2\eta,
    \end{equation*}
    where $\omega_i=\theta_i\nu_1+(1-\theta_i)\omega.$ Then we have $$\int \varphi d\nu_1<\int \varphi d\omega_2=\theta_2\int \varphi d\nu_1+(1-\theta_2)\int \varphi d\omega< \theta_1\int \varphi d\nu_1+(1-\theta_1)\int \varphi d\omega=\int \varphi d\omega_1,$$ 
    and $S_{\omega_1}=S_{\omega_2}=S_\omega=X,$ $\omega_1(\bigcup_{n\geq 1}X_{n})=\omega_2(\bigcup_{n\geq 1}X_{n})=1.$
    Let 
    \begin{align*}
    	K_1\ &:=\{\tau\nu_1+(1-\tau)\nu_2:\tau\in[0,1]\},\\
    	K_2\ &:=\{\tau\nu_1+(1-\tau)\omega_1:\tau\in[0,1]\},\\
    	K_2\ &:=\{\tau\omega_1+(1-\tau)\omega_2:\tau\in[0,1]\}.
    \end{align*}

    By Theorem \ref{maintheorem-3}(2), we have $$h_{top}(f,G_{K_1}\cap U\cap Trans\})=\inf\{h_{\nu}(f):\nu\in K_1\}.$$
    Note that $G_{K_1}\subseteq I_{\varphi}(f)\cap \{x\in X:V_T(x)\subseteq \cM_{f|_{X_{n_0}}}(X_{n_0})\},$ we have 
    \begin{equation}\label{equa-AA}
    	\begin{split}
    		&h_{top}(f,I_{\varphi}(f)\cap \{x\in X:V_T(x)\subseteq \cM_{f|_{X_{n_0}}}(X_{n_0})\}\cap U\cap Trans\})\\
    		&\geq h_{top}(f,G_{K_1}\cap U\cap Trans\})\\
    		&=\inf\{h_{\nu}(f):\nu\in K_1\}\\
    		&=\inf\{\tau h_{\nu_1}(f)+(1-\tau)h_{\nu_{2}}(f):\tau\in [0,1]\}\\
    		&\geq h_\mu(f)-\eta\\
    		&\geq h_{top}(f,X_{n_0})-2\eta.
    	\end{split}
    \end{equation}
    By the arbitrariness of $\eta,$ this yields $$h_{top}(f,I_{\varphi}(f)\cap \{x\in X:V_T(x)\subseteq \cM_{f|_{X_{n_0}}}(X_{n_0})\}\cap U\cap Trans\})\geq h_{top}(f,X_{n_0}).$$
    
    If $(X,f)$ has uniform separation property, then by Theorem \ref{maintheorem-3}(1), we have $$h_{top}(f,G_{K_2}\cap U\cap Trans\})=\inf\{h_{\nu}(f):\nu\in K_2\},$$ and $$h_{top}(f,G_{K_3}\cap U\cap Trans\})=\inf\{h_{\nu}(f):\nu\in K_3\}.$$ Note that $$G_{K_2}\subseteq I_{\varphi}(f)\cap \{x\in X: \exists\ \mu_1,\mu_{2}\in V_f(x)\text{ s.t. }S_{\mu_1}\subseteq X_{n_0}, S_{\mu_{2}}=X \},$$ and $$G_{K_3}\subseteq I_{\varphi}(f)\cap \{x\in X:S_\mu=X \text{ for any }\mu \in V_T(x)\}.$$ In a similar manner of  (\ref{equa-AA}), we have
    $$h_{top}(f,I_{\varphi}(f)\cap \{x\in X: \exists\ \mu_1,\mu_{2}\in V_f(x)\text{ s.t. }S_{\mu_1}\subseteq X_{n_0}, S_{\mu_{2}}=X \}\cap U\cap Trans\})\geq h_{top}(f,X_{n_0})-3\eta,$$
    $$h_{top}(f,I_{\varphi}(f)\cap \{x\in X:S_\mu=X \text{ for any }\mu \in V_T(x)\}\cap U\cap Trans\})\geq h_{top}(f,X_{n_0})-3\eta.$$
    By the arbitrariness of $\eta,$ this yields 
    $$h_{top}(f,I_{\varphi}(f)\cap \{x\in X: \exists\ \mu_1,\mu_{2}\in V_f(x)\text{ s.t. }S_{\mu_1}\subseteq X_{n_0}, S_{\mu_{2}}=X \}\cap U\cap Trans\})\geq h_{top}(f,X_{n_0}),$$
    $$h_{top}(f,I_{\varphi}(f)\cap \{x\in X:S_\mu=X \text{ for any }\mu \in V_T(x)\}\cap U\cap Trans\})\geq h_{top}(f,X_{n_0}).$$
    
    Now we prove that the functions with  $I_{\varphi}(f) \cap X_{n_0}\neq\emptyset$  are open and dense in $C^{0}(X,\mathbb{R})$. Since $(X_{n_0},f)$ is transitive and has the shadowing property, and $X_{n_o}$ is not periodic, then by Lemma \ref{BD} it has positive topological entropy. By \cite[Theorem 1.5]{DOT} the set of irregular points $\cup_{\phi\in C^{0}(X_{n_0},\mathbb{R})} I_\phi(f|_{X_{n_0}})$ is not empty and carries same topological entropy as $f|_{X_{n_0}}$, in particular  there exists some $\phi\in C^{0}(X_{n_0},\mathbb{R})$ with $I_{\phi}(f|_{X_{n_0}})\neq \emptyset.$ By Proposition \ref{proposition-AD}, the proof is complete.
\end{proof}
\begin{Rem}
	From the definition of uniform separation property, if $(X,f)$ has uniform separation property, then so does $(X_n,f|_{X_n})$ for any $n\in\mathbb{N^{+}}.$
\end{Rem}
From the proof of Theorem \ref{maintheorem-6}, it is easy to see that if we replace $I_{\varphi}(f) \cap X_{n_0}\neq\emptyset$ by $\mathrm{Int}(L_\varphi|_{X_{n_0}})\neq \emptyset,$ the result also holds. So we have the following.
\begin{Cor}\label{Coro-1}
	Suppose that $f$ is a homeomorphism from $X$ onto $X$,  $X$ is not periodic, $(X,f)$ has a sequence of nondecreasing $f$-invariant compact subsets $\{X_{n} \subseteq X:n \in \mathbb{N^{+}} \}$ such that $\overline{\bigcup_{n\geq 1}X_{n}}=X$, $\mathrm{Per}(f|_{X_{1}})$ $\neq \emptyset,$ $({X_{n}},f|_{X_{n}})$ has shadowing property, uniform separation property and is transitive for any $n \in \mathbb{N^{+}}.$ If for any $\eta>0,$ any  $\mu \in \mathcal{M}_{f}(X)$ and any neighbourhood $G$ of $\mu$ in $\mathcal{M}(X)$, there exist $n\in\mathbb{N^{+}}$ and $\nu \in \mathcal{M}_{f|_{X_n}}(X_n)$ such that $\nu\in G$ and $h_{\nu}(f)\geq h_{\mu}(f)-\eta.$
	Then for any $\varphi \in C^{0}(X,\mathbb{R})$ with $I_{\varphi}(f) \neq\emptyset,$ there exists $N\in\mathbb{N^{+}}$ such that for any non-empty open set $U\subseteq X,$ we have 
	$$h_{top}(f,I_{\varphi}(f)\cap \{x\in X:V_T(x)\subseteq \cM_{f|_{X_{n}}}(X_{n})\text{ for some }n\geq N\}\cap U\cap Trans\})= h_{top}(f).$$
	Moreover, if $(X,f)$ has uniform separation property, then we have
	$$h_{top}(f,I_{\varphi}(f)\cap \{x\in X: \exists\ \mu_1,\mu_{2}\in V_f(x)\text{ s.t. }S_{\mu_1}\subseteq X_{n}, S_{\mu_{2}}=X \text{ for some }n\geq N\}\cap U\cap Trans\})= h_{top}(f),$$
	$$h_{top}(f,I_{\varphi}(f)\cap \{x\in X:S_\mu=X \text{ for any }\mu \in V_T(x)\}\cap U\cap Trans\})= h_{top}(f).$$
	Moreover, the functions with  $I_{\varphi}(f)\neq\emptyset$ are open and dense in the space of continuous functions. 
\end{Cor}
\begin{proof}
	For any $\eta>0,$ any  $\mu \in \mathcal{M}_{f}(X)$ and any neighbourhood $G$ of $\mu$ in $\mathcal{M}(X)$, there exist $m\in\mathbb{N^{+}}$ and $\nu \in \mathcal{M}_{f|_{X_m}}(X_m)$ such that $\nu\in G$ and $h_{\nu}(f)\geq h_{\mu}(f)-\eta.$ Then $\sup\limits_{n\geq 1}h_{top}(f,X_n)\geq h_{top}(f,X_m)\geq h_{\nu}(f)\geq h_{\mu}(f)-\eta.$
	By the arbitrariness of $\mu$ and $\eta,$ this yields $\sup\limits_{n\geq 1}h_{top}(f,X_n)= h_{top}(f,X).$
	
	Since $I_{\varphi}(f)\neq\emptyset,$ one has  $\mathrm{Int}(L_\varphi)\neq\emptyset$, then there exist $\lambda_1,\lambda_2\in\mathcal M_{f}(X)$ and $a,b\in\mathbb{R}$ such that $\int\varphi d\lambda_1<a<b< \int\varphi d\lambda_2.$ Then there exist $N\in\mathbb{N^{+}}$ and $\nu_1,\nu_2 \in \mathcal{M}_{f|_{X_N}}(X_N)$ such that $\int\varphi d\nu_1<\frac{a+b}{2}< \int\varphi d\nu_2.$ This implies $\mathrm{Int}(L_\varphi|_{X_{n}})\neq\emptyset$ for any $n\geq N.$ So by Theorem \ref{maintheorem-6} we have
	\begin{equation}\label{equa-AD}
		\begin{split}
			&h_{top}(f,I_{\varphi}(f)\cap \{x\in X:V_T(x)\subseteq \cM_{f|_{X_{n}}}(X_{n})\text{ for some }n\geq N\}\cap U\cap Trans\})\\
			\geq &\sup\limits_{n\geq N}h_{top}(f,X_n)
			=\sup\limits_{n\geq 1}h_{top}(f,X_n)
			= h_{top}(f).
		\end{split}
	\end{equation}
    If further $(X,f)$ has uniform separation property, then
    in a similar manner of  (\ref{equa-AD}), we have $$h_{top}(f,I_{\varphi}(f)\cap \{x\in X: \exists\ \mu_1,\mu_{2}\in V_f(x)\text{ s.t. }S_{\mu_1}\subseteq X_{n}, S_{\mu_{2}}=X \text{ for some }n\geq N\}\cap U\cap Trans\})= h_{top}(f),$$
    $$h_{top}(f,I_{\varphi}(f)\cap \{x\in X:S_\mu=X \text{ for any }\mu \in V_T(x)\}\cap U\cap Trans\})= h_{top}(f).$$
	This complete tha proof of Corollary \ref{Coro-1}.
\end{proof}

\begin{maintheorem}\label{maintheorem-7}
	Suppose that $f$ is a homeomorphism from $X$ onto $X$, $(X,f)$ has a sequence of nondecreasing $f$-invariant compact subsets $\{X_{n} \subseteq X:n \in \mathbb{N^{+}} \}$ such that $\overline{\bigcup_{n\geq 1}X_{n}}=X$, $\mathrm{Per}(f|_{X_{1}})$ $\neq \emptyset,$ $({X_{n}},f|_{X_{n}})$ has shadowing property and is transitive for any $n \in \mathbb{N^{+}}.$ 
	\begin{description}
		\item[(1)] If $(X_{n_0},f|_{X_{n_0}})$ has uniform separation property for some ${n_0}\in\mathbb{N^{+}}$, then for any non-empty open set $U\subseteq X,$ we have 
		$$h_{top}(f,QR(f)\cap \{x\in X:V_T(x)\text{ consists of one measure in } \cM_{f|_{X_{n_0}}}(X_{n_0})\}\cap U\cap Trans\})\geq h_{top}(f,X_{n_0}),$$ 
		\item[(2)] If $(X,f)$ has uniform separation property, then for any non-empty open set $U\subseteq X$ and any $n\geq 1,$ we have 
		$$h_{top}(f,QR(f)\cap \{x\in X:V_T(x)\text{ consists of one measure in } \cM_{f|_{X_{n}}}(X_{n})\}\cap U\cap Trans\})\geq h_{top}(f,X_{n}),$$
		$$h_{top}(f,QR(f)\cap \{x\in X:V_T(x)\text{ consists of one measure with full support}\}\cap U\cap Trans\})\geq \sup\limits_{n\geq 1}h_{top}(f,X_n).$$
	\end{description}
\end{maintheorem}
\begin{proof}
	(1) For any $\eta>0,$ by the variational principle, there is $\mu\in \cM_{f|_{X_{n_0}}}(X_{n_0})$ such that $h_{\mu}(f)\geq h_{top}(f,X_{n_0})-\eta.$ 
	By Theorem \ref{maintheorem-3}(2), we have $$h_{top}(f,G_{\mu}\cap U\cap Trans\})=h_\mu(f).$$ Note that $G_{\mu}\subseteq QR(f)\cap \{x\in X:V_T(x)\text{ consists of one measure in } \cM_{f|_{X_{n_0}}}(X_{n_0})\},$ we have 
	\begin{equation*}
		\begin{split}
			&h_{top}(f,QR(f)\cap \{x\in X:V_T(x)\text{ consists of one measure in } \cM_{f|_{X_{n_0}}}(X_{n_0})\}\cap U\cap Trans\})\\
			\geq& h_\mu(f)
			\geq h_{top}(f,X_{n_0})-\eta.
		\end{split}
	\end{equation*}
    By the arbitrariness of $\eta,$ this yields $$h_{top}(f,QR(f)\cap \{x\in X:V_T(x)\text{ consists of one measure in } \cM_{f|_{X_{n_0}}}(X_{n_0})\}\cap U\cap Trans\})\geq h_{top}(f,X_{n_0}),$$ which completes the proof of item (1).
    
    (2) Note that if $(X,f)$ has uniform separation property, then so does $(X_n,f|_{X_n})$ for any $n\in\mathbb{N^{+}}.$ Thus we have 
    $$h_{top}(f,QR(f)\cap \{x\in X:V_T(x)\text{ consists of one measure in } \cM_{f|_{X_{n}}}(X_{n})\}\cap U\cap Trans\})\geq h_{top}(f,X_{n}),$$ by item (1).
    
    For any $n\geq 1$ and any $\eta>0,$ by the variational principle, there is $\mu\in \cM_{f|_{X_{n}}}(X_{n})$ such that $h_{\mu}(f)\geq h_{top}(f,X_{n})-\eta.$ 
    By Proposition \ref{BG}, there exists $\omega\in \mathcal{M}_{f}(X)$ with $S_\omega=X$ and $\omega(\bigcup_{n\geq 1}X_{n})=1.$ Take $0< \theta<1$ close to 1 such that
    \begin{equation*}
    	h_{\nu} (f)=\theta h_{\mu}(f)+(1-\theta)h_{\omega}(f) \geq h_{\mu}(f)-\eta,
    \end{equation*}
    where $\nu=\theta\mu+(1-\theta)\omega.$ Then we have $S_{\nu}=S_\omega=X$ and $\nu(\bigcup_{n\geq 1}X_{n})=1.$
    By Theorem \ref{maintheorem-3}(1), we have $$h_{top}(f,G_{\nu}\cap U\cap Trans\})=h_\nu(f).$$ Note that $G_{\mu}\subseteq QR(f)\cap \{x\in X:V_T(x)\text{ consists of one measure with full support}\},$ we have 
    \begin{equation*}
    	\begin{split}
    		&h_{top}(f,QR(f)\cap \{x\in X:V_T(x)\text{ consists of one measure with full support}\}\cap U\cap Trans\})\\
    		\geq &h_\nu(f)
    		\geq h_{top}(f,X_{n})-2\eta.
    	\end{split}
    \end{equation*}
    By the arbitrariness of $n$ and $\eta,$ this yields $$h_{top}(f,QR(f)\cap \{x\in X:V_T(x)\text{ consists of one measure with full support}\}\cap U\cap Trans\})\geq \sup\limits_{n\geq 1}h_{top}(f,X_n),$$ which completes the proof of item (2).
\end{proof}

In a similar manner of Corollary \ref{Coro-1}, we have the following.

\begin{Cor}
	Suppose that $f$ is a homeomorphism from $X$ onto $X$,  $(X,f)$ has a sequence of nondecreasing $f$-invariant compact subsets $\{X_{n} \subseteq X:n \in \mathbb{N^{+}} \}$ such that $\overline{\bigcup_{n\geq 1}X_{n}}=X$, $\mathrm{Per}(f|_{X_{1}})$ $\neq \emptyset,$ $({X_{n}},f|_{X_{n}})$ has shadowing property, uniform separation property and is transitive for any $n \in \mathbb{N^{+}}.$ If for any $\eta>0,$ any  $\mu \in \mathcal{M}_{f}(X)$ and any neighbourhood $G$ of $\mu$ in $\mathcal{M}(X)$, there exist $n\in\mathbb{N^{+}}$ and $\nu \in \mathcal{M}_{f|_{X_n}}(X_n)$ such that $\nu\in G$ and $h_{\nu}(f)\geq h_{\mu}(f)-\eta.$
	Then there exists $N\in\mathbb{N^{+}}$ such that  for any non-empty open set $U\subseteq X,$ we have 
	$$h_{top}(f,QR(f)\cap \{x\in X:V_T(x)\text{ consists of one measure in } \bigcup_{n\geq N}\cM_{f|_{X_{n}}}(X_{n})\}\cap U\cap Trans\})= h_{top}(f).$$
	Moreover, if $(X,f)$ has uniform separation property, then for any non-empty open set $U\subseteq X,$ we have 
	$$h_{top}(f,QR(f)\cap \{x\in X:V_T(x)\text{ consists of one measure with full support}\}\cap U\cap Trans\})= h_{top}(f).$$
\end{Cor}

\subsection{Proofs of Theorem \ref{maintheorem-1'} and \ref{maintheorem-1}}

First, we give the proof of Theorem \ref{maintheorem-1'}.

\noindent\textbf{Proof of Theorem \ref{maintheorem-1'}:} By definition, we can take a sequense of hyperbolic periodic points $\{p_n\}$ dense in $H(p)$ and all homoclinically related to $p$ such that $p_{1}=p$ and $p_2=q$.
Then we can construct an increasing sequence of hyperbolic horseshoes $\{\Lambda_n\}_{n\geq 1}$ contained in $H(p)$ inductively as follows:
\begin{itemize}
		\item Let $\Lambda_1=\Lambda.$
	\item Let $\Lambda_2$ be a transitive locally maximal hyperbolic set that contains $p_1$ and $\Lambda$. Such $\Lambda_1$ exists since $p_1$ is homoclinically related to $p_2$ (for example, see \cite[Lemma 8]{Newho1979}).
	\item For $n\geq 3$, let $\Lambda_n$ be a transitive locally maximal hyperbolic that contains $p_{n}$ and also contains the  hyperbolic set $\Lambda_{n-1}$.
\end{itemize}
By density of  $\{p_n\}$, the union of $\Lambda_n$ is dense in $H(p)$. Since $\Lambda_n$ is a locally maximal hyperbolic set, then it is expansive by  \cite[Corollary 6.4.10]{KatHas} and has shadowing property by \cite[Theorem 18.1.2]{KatHas}. Thus $\Lambda_n$ has uniform separation property by Proposition \ref{proposition-AA}.  By Lemma \ref{BD}, one has $h_{top}(f,\Lambda)>0.$
Since $H(p)$ is not uniformly hyperbolic, then $x$ is not uniformly hyperbolic if $\omega_f(x)=H(p).$ Note that for any $n\in\mathbb{N^{+}},$ each $\mu\in\mathcal{M}_{f|_{\Lambda_n}}(\Lambda_n)$ is uniformly hyperblic, and each $\nu\in\mathcal{M}_f(H(p))$ with $S_{\nu}=H(p)$ is not uniformly hperbolic.  So we have Theorem \ref{maintheorem-1'} by Theorem \ref{maintheorem-6} and Theorem \ref{maintheorem-7}.\qed

Next, we give the proof of Theorem \ref{maintheorem-1}.

\noindent\textbf{Proof of Theorem \ref{maintheorem-1}:}  
Since $\mu$ is a hyperbolic ergodic measure, for any $\eta>0,$
from \cite{Katok} (or Theorem S.5.3 on Page 694 of book \cite{KatHas}), there exist a hyperbolic periodic point $p$ and a transitive locally maximal hyperbolic set $\Lambda$ such that $S_\mu\subset H(p),$ $p\in \Lambda \subset H(p)$ and $h_{top}(f,\Lambda)\geq h_{\mu}(f)-\eta.$
Then by Theorem \ref{maintheorem-1'}, we have 
\begin{description}
	\item[(a)] $h_{top}(f,\{x\in M: x\text{ is not uniformly hyperbolic, }\text{each }\mu\in V_f(x) \text{ is uniformly hyperbolic}\}\}\cap IR(f)\cap Rec)\geq h_{top}(f,\Lambda)\geq h_{\mu}(f)-\eta,$
	\item[(b)] $h_{top}(f,\{x\in M: x\text{ is not uniformly hyperbolic, }V_f(x) \text{ consists of one uniformly hyperbolic me-}\\ \text{asure}\}\cap QR(f)\cap Rec)\geq h_{top}(f,\Lambda)\geq h_{\mu}(f)-\eta.$
\end{description} 
Moreover, if $H(p)$ has uniform separation property, then
\begin{description}
	\item[(c)] $h_{top}(f,\{x\in M: x\text{ is not uniformly hyperbolic, }\exists\ \mu_1,\mu_{2}\in V_f(x)\text{ s.t. } \mu_1\text{ is uniformly hyperbolic,}\\ \mu_2\text{ is not uniformly hyperbolic}\}\cap IR(f)\cap Rec)\geq h_{top}(f,\Lambda)\geq h_{\mu}(f)-\eta,$
	\item[(d)] $h_{top}(f,\{x\in M: x\text{ is not uniformly hyperbolic, }\text{each }\mu\in V_f(x) \text{ is not uniformly hyperbolic}\}\cap IR(f)\cap Rec)\geq h_{top}(f,\Lambda)\geq h_{\mu}(f)-\eta,$
	\item[(e)] $h_{top}(f,\{x\in M: x\text{ is not uniformly hyperbolic, }V_f(x) \text{ consists of one measure which is not uni-}\\ \text{formly hyperbolic }\}\cap QR(f)\cap Rec)\geq h_{top}(f,\Lambda)\geq h_{\mu}(f)-\eta.$
\end{description} 
By the arbitrariness of $\eta,$ we complete the proof.
\qed

\section*{ Acknowledgements.}{ 
	X. Hou and X. Tian are 
	supported by National Natural Science Foundation of China (grant No.   12071082,11790273) and in part by Shanghai Science and Technology Research Program (grant No. 21JC1400700).
}

\end{document}